\documentclass[11pt, amsfonts]{amsart}

\usepackage{amsmath,amssymb,amsthm,enumerate,comment}
\usepackage[all]{xy}
\usepackage{xcolor}
\usepackage{graphicx}
\usepackage[T1]{fontenc}
\usepackage[dvipdfm,colorlinks=true]{hyperref}
\usepackage[colorlinks=true]{hyperref}

\SelectTips{eu}{12}

\newtheorem{theorem}{Theorem}[section]
\newtheorem{lemma}[theorem]{Lemma}
\newtheorem{proposition}[theorem]{Proposition}
\newtheorem{corollary}[theorem]{Corollary}

\theoremstyle{definition}
\newtheorem{definition}[theorem]{Definition}

\newtheorem{construction}[theorem]{Construction}

\theoremstyle{remark}
\newtheorem{remark}[theorem]{Remark}

\let\latexchi\chi
\makeatletter
\renewcommand\chi{\@ifnextchar_\sub@chi\latexchi}
\newcommand{\sub@chi}[2]{
  \@ifnextchar^{\subsup@chi{#2}}{\latexchi^{}_{#2}}%
}
\newcommand{\subsup@chi}[3]{
  \latexchi_{#1}^{#3}%
}
\makeatother

\makeatletter
\@addtoreset{equation}{section} 
\makeatother

\makeatletter
\newsavebox{\@brx}
\newcommand{\llangle}[1][]{\savebox{\@brx}{\(\m@th{#1\langle}\)}%
  \mathopen{\copy\@brx\kern-0.5\wd\@brx\usebox{\@brx}}}
\newcommand{\rrangle}[1][]{\savebox{\@brx}{\(\m@th{#1\rangle}\)}%
  \mathclose{\copy\@brx\kern-0.5\wd\@brx\usebox{\@brx}}}
\makeatother

\DeclareMathOperator{\id}{id}
\DeclareMathOperator{\Ker}{Ker}
\DeclareMathOperator{\Hom}{Hom}

\DeclareMathOperator{\scl}{scl}
\DeclareMathOperator{\cl}{cl}

\DeclareMathOperator{\Aut}{Aut}
\DeclareMathOperator{\PSL}{PSL}
\DeclareMathOperator{\Sp}{Sp}
\DeclareMathOperator{\Homeo}{Homeo}

\newcommand{\NN}{\mathbb{N}}
\newcommand{\RR}{\mathbb{R}}
\newcommand{\ZZ}{\mathbb{Z}}

\newcommand{\GG}{\Gamma}

\newcommand{\genus}{\ell}
\newcommand{\G}{G_{\genus}}

\newcommand{\R}{R_{\genus}}

\newcommand{\Gps}{G_{\psi}}
\newcommand{\Nps}{N_{\psi}}
\newcommand{\Gamps}{\Gamma_{\psi}}
\newcommand{\pps}{p_{\psi}}
\newcommand{\A}{A}

\newcommand{\h}{\mathcal{H}}
\renewcommand{\th}{\widetilde{\h}}
\newcommand{\thr}{\th_{\RR}}
\newcommand{\tf}{\widetilde{f}}
\newcommand{\tg}{\widetilde{g}}

\newcommand{\trho}{\widetilde{\rho}}

\newcommand{\talpha}{\widetilde{\alpha}}
\newcommand{\tbeta}{\widetilde{\beta}}
\newcommand{\try}{\widetilde{\rho_{\genus}(y)}}
\newcommand{\trrz}{\widetilde{\rho_{\genus}(z)}}

\newcommand{\taur}{\tau_{\RR}}
\newcommand{\QQQ}{\mathrm{Q}}
\newcommand{\HHH}{\mathrm{H}}
\newcommand{\Symp}{\mathrm{Symp}}
\newcommand{\Ham}{\mathrm{Ham}}
\newcommand{\basic}{\mathrm{basic}}
\newcommand{\W}{\mathrm{W}}

\makeatletter
\@namedef{subjclassname@2020}{%
  \textup{2020} Mathematics Subject Classification}
\makeatother

\title[Invariant quasimorphisms and non-equivalence of $\scl$]{Invariant quasimorphisms for groups acting on the circle and non-equivalence of SCL}

\author[S. Maruyama]{Shuhei Maruyama}
\address[Shuhei Maruyama]{Graduate School of Mathematics, Nagoya University, Furocho, Chikusaku, Nagoya, 464-8602, Japan}
\email{m17037h@math.nagoya-u.ac.jp}

\author[T. Matsushita]{Takahiro Matsushita}
\address[Takahiro Matsushita]{Department of Mathematical Sciences, University of the Ryukyus, Nishihara-cho, Okinawa 903-0213, Japan}
\email{mtst@sci.u-ryukyu.ac.jp}

\author[M. Mimura]{Masato Mimura}
\address[Masato Mimura]{Mathematical Institute, Tohoku University, 6-3, Aramaki Aza-Aoba, Aoba-ku, Sendai 9808578, Japan}
\email{m.masato.mimura.m@tohoku.ac.jp}

\begin{document}



\begin{abstract}
  We construct invariant quasimorphisms for groups acting on the circle.
  Furthermore, we provide a criterion for the non-extendablity of the resulting quasimorphisms and an explicit formula which relates the values of our quasimorphisms to those of the Poincar\'{e} translation number.
  By using them, we show that the stable commutator length $\scl_G$ and the stable mixed commutator length $\scl_{G,N}$ are not bi-Lipschitzly equivalent for the surface group $G=\pi_1(\Sigma_{\ell})$ of genus at least $2$ and its commutator subgroup $N = [\pi_1(\Sigma_{\ell}), \pi_1(\Sigma_{\ell})]$.
  We also show the non-equivalence for a pair $(G,N)$ such that $G$ is the fundamental group of a $3$-dimensional closed hyperbolic mapping torus.
  These pairs serve as the first family of examples of such $(G,N)$ in which $G$ is finitely generated.

\end{abstract}

\maketitle

\section{Introduction}
\subsection{Invariant quasimorphisms}\label{subsec:intro_inv_qm}
A real-valued function $\mu \colon G \to \RR$ on a group $G$ is called a \emph{quasimorphism} if
\[
  D(\mu) = \sup_{g,h \in G} |\mu(gh) - \mu(g) - \mu(h)| < \infty.
\]
A quasimorphism $\mu$ is said to be \emph{homogeneous} if $\mu(g^n) = n \cdot \mu(g)$ for every $g \in G$ and $n \in \ZZ$.
Let $\QQQ(G)$ denote the real vector space of homogeneous quasimorphisms on $G$.
Homogeneous quasimorphisms are studied in many branch of mathematics, such as dynamical system, geometric group theory, and symplectic geometry (see \cite{PR}, \cite{Cal}).

For a normal subgroup $N$ of $G$, a homogeneous quasimorphism $\mu$ on $N$ is said to be \emph{$G$-invariant} if
\[
  \mu(gxg^{-1}) = \mu(x)
\]
for every $g \in G$ and $x \in N$.
Let $\QQQ(N)^G$ denote the real vector space of $G$-invariant homogeneous quasimorphisms on $N$.
Invariant homogeneous quasimorphisms have been constructed in symplectic geometry and geometric group theory (e.g., \cite{EP03}, \cite{Oh05}, \cite{MR3524783} for symplectic geometric constructions and \cite{BM}, \cite{Karlh}, \cite{Karlh2}, \cite{FW} for geometric group theoretic ones; see also \cite{2212.11180}).
Since every homogeneous quasimorphism on $G$ is conjugation invariant, its restriction to $N$ is a $G$-invariant homogeneous quasimorphism on $N$.
That is, the pullback $i^* \colon \QQQ(G) \to \QQQ(N)$ by the inclusion $i \colon N \to G$ factors through $\QQQ(N)^G$.
This is one of the sources of $G$-invariant homogeneous quasimorphisms on $N$.

A $G$-invariant homogeneous quasimorphism on $N$ is said to be \emph{non-extendable} if it is \emph{not} a restriction of a homogeneous quasimorphism on $G$.
In this paper, we provide a method to construct non-extendable $G$-invariant homogeneous quasimorphisms on $N$.
Set $\GG = G/N$ and let $p \colon G \to \GG$ be the projection.
Let
\begin{align*}
  \h = \Homeo_+(S^1)
\end{align*}
be the group of orientation preserving homeomorphisms of the circle, $e_{\RR} \in \HHH^2(\h;\RR)$ the real Euler class, and $\chi_b$ the canonical Euler cocycle.
Let $G$ be a group acting on the circle $S^1$ and $\rho \colon G \to \h$ the representation induced from the action.
Under the assumption that the pullback $\rho^* e_{\RR} \in \HHH^2(G;\RR)$ is contained in the image of $p^* \colon \HHH^2(\GG;\RR) \to \HHH^2(G;\RR)$, we construct invariant homogeneous quasimorphisms
\[
  \nu_{\rho, A, u} \in \QQQ(N)^G,
\]
where $A$ is a group $2$-cochain on $\GG$ and $u$ is a group $1$-cochain on $G$ satisfying $\rho^* \chi_b - p^* A = \delta u$ (Construction \ref{construction}).

Related to the non-extendability, we set
\begin{align}\label{sp_non-ext}
  \W(G,N) = \QQQ(N)^G / (\HHH^1(N;\RR)^G + i^*\QQQ(G)),
\end{align}
where $\HHH^1(N;\RR)^G$ is the vector space of real-valued homomorphisms on $N$ which are invariant under $G$-conjugation.
This space $\W(G,N)$ naturally appears in the context of the comparison of stable commutator length (see Subsection \ref{subsec:intro_app}).
Note that, if a $G$-invariant homogeneous quasimorphism on $N$ is non-zero in $\W(G,N)$, then it is non-extendable to $G$.
To the best of the authors' knowledge, the following example is the essentially only one concrete example of $G$-invariant homogeneous quasimorphism that is non-zero in $\W(G,N)$.
Let $\Sigma_{\genus}$ be a closed connected oriented surface of genus $\genus \geq 2$ and $\omega$ a symplectic form on it.
Let $G$ be the identity component $\Symp_0(\Sigma_{\genus},\omega)$ of the symplectomorphism group of $\Sigma_{\genus}$ and $N$ the Hamiltonian diffeomorphism group $\Ham(\Sigma_{\genus}, \omega)$.
Kawasaki and Kimura showed in \cite{MR4425357} that Py's Calabi quasimorphism, which is an element of $\QQQ(N)^G$, is non-zero in $\W(G,N)$.

The space $\W(G,N)$ is studied in \cite{non-extendable}, and the dimension of $\W(G,N)$ is determined for certain pairs of groups.
For instance, if $G$ is the fundamental group $\pi_1(\Sigma_\genus)$ of the closed connected oriented surface of genus $\genus \geq 2$ and $N$ is the commutator subgroup, then $\dim \W(G,N) = 1$.
However, since the dimension of $\W(G,N)$ is determined in a homological manner, the explicit representatives of the non-zero element in $\W(G,N)$ were unclear at that point.
The above $\nu_{\rho, A, u}$ provides such a representative as follows.
\begin{theorem}[criterion of the non-extendability]\label{thm:intro_non-ext}
  Assume that the second bounded cohomology group $\HHH_b^2(\GG;\RR)$ is trivial.
  If the Euler class $\rho^*e_{\RR}$ is non-zero, then $\nu_{\rho, A, u}$ is \emph{non-zero} in $\W(G,N)$.
  In particular, if $\rho^*e_{\RR}$ is non-zero, then $\nu_{\rho, A, u}$ is \emph{non-extendable} to $G$.
\end{theorem}
Note that if $\GG$ is amenable (e.g., abelian, nilpotent, solvable), then the bounded cohomology groups $\HHH_b^n(\GG;\RR)$ are trivial for $n \geq 1$.

In the case that $G = \pi_1(\Sigma_{\genus})$ of genus $\genus \geq 2$ and its commutator subgroup $N$, we take $\rho \colon G \to \h$ as a Fuchsian representation for instance.
Then the $G$-invariant homogeneous quasimorphism $\nu_{\rho, A, u}$ provides a representative of the unique element of $\W(G,N)$ up to non-zero constant multiple.

Furthermore, we study numerical aspects of $\nu_{\rho, A, u}$.
More precisely, we show that $D(\nu_{\rho, A, u}) \leq 1$ and obtain an explicit formula connecting the value $\nu_{\rho, A, u}(z)$ for $z \in [G,N]$ and the Poincar\'{e} translation number (see Theorem \ref{thm:formula}).
Such numerical pieces of information may be useful when we apply the Bavard duality theorem (see Theorem \ref{thm:mixedBavard}) to certain elements of $[G,N]$; in fact, we employ them to show the non-equivalence of the stable commutator length, which we will explain in the next subsection.

We also obtain a rigidity result on $\nu_{\rho, A, u}$ with respect to semi-conjugacy.
We set
\[
  \Hom(G,\h)_{\basic} = \{ \rho \colon G \to \h \colon \text{a homomorphism with } \rho^*e_{\RR} \in p^*(\HHH^2(\GG;\RR)) \}.
\]
If we assume that $N$ is contained in the commutator subgroup of $G$, then Construction \ref{construction} gives rise to a well-defined map
\[
  \Phi \colon \Hom(G,\h)_{\basic} \to \QQQ(N)^G/\HHH^1(N;\RR)^G.
\]
In particular, the class $\Phi(\rho) = [\nu_{\rho, A, u}] \in \QQQ(N)^G/\HHH^1(N;\RR)^G$ is independent on the choice of $A$ and $u$.
It will be shown that the map $\Phi$ is \emph{rigid} under semi-conjugacy (Proposition \ref{prop:rigidity}).

\subsection{Application to the non-equivalence of $\mathrm{SCL}$}\label{subsec:intro_app}

By using the above invariant homogeneous quasimorphisms, we provide the first family of examples of \emph{finitely generated} group $G$ and its normal subgroup $N$ for which $\scl_G$ and $\scl_{G,N}$ are \emph{not} equivalent.

A \emph{\textup{(}single\textup{)} mixed commutator} is an element of the form $[g, w] = gwg^{-1}w^{-1}$ with $g \in G$ and $w \in N$.
Let $[G, N]$ be the subgroup of $G$ generated by mixed commutators, which we call the \emph{mixed commutator subgroup}.
For an element $x \in [G, N]$, the \emph{mixed commutator length} $\cl_{G, N}(x)$ is the least number of mixed commutators needed to express $x$ as their product.

\begin{definition}[stable mixed commutator length]\label{def:mixed_scl}
Let $G$ be a group, $N$ its subgroup, and $\GG$ the quotient group $G/N$.
The \emph{stable mixed commutator length} $\scl_{G, N}$ is defined as the function
\[
\scl_{G, N}\colon [G,N] \to \RR_{\geq 0};\quad x\mapsto \lim_{m \to \infty} \frac{\cl_{G, N}(x^m)}{m}.
\]
\end{definition}
It follows from Fekete's lemma that the limit above always exists.
If $N = G$, then $[G,G]$, $\cl_{G,G}$ and $\scl_{G,G}$ coincide with the commutator subgroup $G'$, the commutator length $\cl_G$ and the stable commutator length $\scl_G\colon G'\to \RR_{\geq 0}$, respectively.
We study the \emph{equivalence problem} of $\scl_G$ and $\scl_{G,N}$, which asks whether $\scl_G$ and $\scl_{G,N}$ are \emph{equivalent} in the following sense.

\begin{definition}[equivalence of $\scl_G$ and $\scl_{G,N}$]\label{def:equiv}
Let $G$ be a group and $N$ its normal subgroup. We say that $\scl_G$ and $\scl_{G,N}$ are \emph{equivalent} if there exists a positive real constant $C$ such that for every $x \in [G,N]$,
\[
 \scl_{G,N}(x) \leq C \cdot \scl_{G}(x)
\]
holds.
\end{definition}
By definition, the inequality $\scl_{G} \leq \scl_{G,N}$ always holds on $[G,N]$. Hence, $\scl_{G}$ and $\scl_{G,N}$ are equivalent (in the sense of Definition~\ref{def:equiv}) if and only if they are bi-Lipschitzly equivalent in the standard sense on $[G,N]$.
(We note that $\scl_{G,N}$ is not defined on $G'\setminus [G,N]$.)

The Bavard duality theorems (\cite{bavard91}, \cite{MR4452430}; see Theorems \ref{thm:Bavard} and \ref{thm:mixedBavard}) state that the equalities
\[
  \scl_{G}(y) = \sup_{[\mu] \in \QQQ(G)/\HHH^1(G;\RR)} \frac{|\mu(y)|}{2D(\mu)}, \hspace{5mm} \scl_{G,N}(x) = \sup_{[\mu] \in \QQQ(N)^G/\HHH^1(N;\RR)^G} \frac{|\mu(x)|}{2D(\mu)}.
\]
hold for every $y \in [G,G]$ and for every $x \in [G,N]$.
Hence it is expected that the difference between $\scl_{G}$ and $\scl_{G,N}$ is caused by the difference of $\QQQ(G)/\HHH^1(G;\RR)$ and $\QQQ(N)^G/\HHH^1(N;\RR)^G$, that is, the space $\W(G,N) = \QQQ(N)^G / (\HHH^1(N;\RR)^G + i^*\QQQ(G))$.
In fact, if $\W(G,N) = 0$, then $\scl_G$ and $\scl_{G,N}$ are equivalent (\cite{non-extendable}, see also Remark~\ref{remark:equivalence}).
In \cite{non-extendable}, it is also shown that if $\HHH^2(G;\RR) = 0$ and $\GG$ is amenable, then the space $\W(G,N)$ is trivial.
For example, if $G$ is either a free group or the fundamental group of a non-orientable surface and $N$ is the commutator subgroup, then the space $\W(G,N)$ is trivial, and in particular, $\scl_G$ and $\scl_{G,N}$ are equivalent.

If $\scl_{G}$ and $\scl_{G,N}$ are not equivalent, then $\W(G,N)$ must contain a non-zero element.
Kawasaki and Kimura used Py's Calabi quasimorphism, which is non-zero in $\W(G,N)$, to show the non-equivalence of $\scl_{G}$ and $\scl_{G,N}$ for $G = \Symp_0(\Sigma_{\genus}, \omega)$ and $\Ham(\Sigma_{\genus}, \omega)$ (\cite{MR4425357}).
This is the first example of such pairs $(G,N)$.
Based on the work in \cite{KKMM2}, we have variants $(G,N)$ of this example (with smaller $G$) in \cite[Example 7.15]{non-extendable}.

By using our representative $\nu_{\rho, A, u}$ of non-zero elements of $\W(G,N)$ in Theorem \ref{thm:intro_non-ext} and its explicit formula (Theorem \ref{thm:formula}), we show the following; this answers the latter question in Problem 9.9 of \cite{non-extendable} in the negative.
\begin{theorem}[non-equivalence for surface groups]\label{thm:main}
  Let $\G$ be the fundamental group of a closed connected oriented surface $\Sigma_{\genus}$ of genus $\genus \geq 2$ and $\G'$ its commutator subgroup.
  Then $\scl_{\G}$ and $\scl_{\G, \G'}$ are \emph{not} equivalent.
\end{theorem}


Furthermore, we provide a pair $(G,N)$ of a $3$-manifold group $G$ and its normal subgroup $N$ with solvable $G/N$ such that $\scl_{G}$ and $\scl_{G, N}$ are not equivalent.
The precise setting goes as follows: let $\genus \geq 2$. Let $X$ be a hyperbolic $3$-manifold that fibers over the circle with fiber $\Sigma_{\genus}$.
Then the fundamental group $\pi_1(X)$ has the presentation
\begin{align*}
\Gps = \G \rtimes_{\psi} \ZZ = \langle \G, c \mid \psi(\gamma) c \gamma^{-1} c^{-1} \rangle
\end{align*}
for some $\psi \in \Aut_+(\G)$. More precisely, $X$ is a mapping torus of some pseudo-Anosov surface diffeomorphism $f_{\psi}$; $f_{\psi}$ induces an action on $\G$, which corresponds to the automorphism $\psi$ above.
Then the abelianization map $\mathrm{Ab}_{\G}\colon\G \to \ZZ^{2\genus}$ induces a surjection
\[
  \pps \colon \Gps \to \ZZ^{2\genus} \rtimes_{s_{\genus}(\psi)} \ZZ.
\]
Here, $s_{\genus} \colon \Aut_+(\G) \to \Sp(2\genus,\ZZ)$ denotes the symplectic representation.
Set $\Nps=\Ker(\pps)$, which equals $\iota(\G')$. Here, $\iota\colon \G \to \Gps = \G \rtimes_{\psi} \ZZ$ is the natural inclusion.

\begin{theorem}[non-equivalence for hyperbolic mapping tori]\label{thm:main3}
For the pair $(\Gps,\Nps)$ in the setting above, $\scl_{\Gps}$ and $\scl_{\Gps, \Nps}$ are \emph{not} equivalent.
\end{theorem}

In the proof of Theorem \ref{thm:main}, we use a (non-standard) action of $\G$ on the circle and a resulting invariant homogeneous quasimorphism.
In the proof of Theorem \ref{thm:main3}, we replace them with a Fuchsian representation of $\G$, the representation extended to $\Gps$, and the resulting invariant homogeneous quasimorphisms (see Subsection \ref{subsec:outline} for more detailed outline of the proofs).

In Section \ref{sec:construction}, we provide the construction of our invariant quasimorphisms (Construction \ref{construction}) mentioned in Subsection \ref{subsec:intro_inv_qm}, as well as the criterion for the non-extendability (Theorem \ref{thm:intro_non-ext}) and the explicit formula (Theorem \ref{thm:formula}).
We employ these results to deduce Theorems \ref{thm:main} and \ref{thm:main3}: we describe the outline of the deduction in Subsection \ref{subsec:outline}.

\subsection{Outlined proofs of Theorems~\ref{thm:main} and \ref{thm:main3}}\label{subsec:outline}

Here we describe the outline of the proofs of Theorems~\ref{thm:main} and \ref{thm:main3}; this description simultaneously provides the organization of the present paper from Section \ref{subsec:overflow}.
Throughout this paper, we use the following symbol for group conjugation
\begin{equation}\label{eq:conj}
{}^h z=hzh^{-1}
\end{equation}
for a group $H$ and for $h,z\in H$.
Recall that we set
\[
\h=\Homeo_+(S^1).
\]
As we already mentioned, this group $\h$ plays a key role for our constructions of quasimorphisms (Subsection~\ref{subsec:poincare} and Section~\ref{sec:construction}).

Let $G$ be a group and $N$ its normal subgroup.
To show the non-equivalence of $\scl_G$ and $\scl_{G,N}$, it suffices to find a sequence $(x_{n} )_{n \in \NN}$
of elements in $[G,N]$
that fulfills the following two conditions:
\begin{enumerate}[(a)]
\item $\sup_{n\in \NN}\scl_{G}(x_n)<\infty$;
\item $\sup_{n\in \NN}\scl_{G,N}(x_n)=\infty$.
\end{enumerate}


To obtain such a sequence $(x_n)_{n\in \NN}$, we employ one auxiliary lemma (Lemma~\ref{lem:hojo}) established in Section \ref{subsec:overflow}, one of whose conditions is stated in terms of a $G$-invariant homogeneous quasimorphism on $N$.


To find an element of the mixed commutator subgroup which applicable to
the auxiliary lemma, we use the group $R_{\genus} = \langle a, b \mid [a, b]^{\genus} \rangle$ for $\genus \geq 2$. 
We consider the elements $y = {}^{ba^2} b^{-1} = ba^2 b^{-1} (ba^2)^{-1}$ and $z = {}^{ba^2} a^{-\genus} = ba^2 a^{-\genus} (ba^2)^{-1}$.
The commutator $[y,z]$ is an element of $[\R, \R']$, and in fact, it is written as a product of $\genus - 1$ mixed commutators (Corollary \ref{cor:relation_in_one_relator_grp}).
These $y$ and $z$ are the source of elements to which we apply
the auxiliary lemma.

Due to the theorem of Eisenbud--Hirsch--Neumann \cite{MR656217} (Theorem~\ref{thm:ehn}), we can construct a representation $\rho_{\genus} \colon \R \to \h$ with non-zero Euler class.
By using this representation, we obtain an $\R$-invariant homogeneous quasimorphism
$\mu_{\rho_{\genus}} = \nu_{\rho_{\genus}, A, u}$ on $\R'$ by Construction \ref{construction}.
By employing Theorem \ref{thm:formula}, we will show that this $\mu_{\rho_{\genus}}$ and the copmmutator expression $[y,z]$
fit in the auxiliary lemma (Lemma~\ref{lem:hojo}); this implies that the sequence $([y,z^n])_{n\in\NN}$ fulfills conditions (a) and (b).
Thus, we obtain the non-equivalence of $\scl_{\R}$ and $\scl_{\R, \R'}$ (Theorem \ref{thm:main2}), which will be shown in Section~\ref{sec:thm2}.

In Section~\ref{sec:thm1}, we show Theorem \ref{thm:main}.
Recall that $\G$ denotes the fundamental group of a closed connected oriented surface of genus $\genus \geq 2$.
The natural projection $q \colon \G \to \R$ induces a surjection $q' \colon \G' \to \R'$ on their commutator subgroups.
Then the pullback $q'^* \mu_{\rho_{\genus}}$ is a $\G$-invariant homogeneous quasimorphism on $\G'$.
By taking lifts $y_i, z_i \in \G$ of $y,z \in \R$ appropriately (the precise form is given in \eqref{eq:y_i,z_i}),
we can apply the auxiliary lemma to
$q'^* \mu_{\rho_{\genus}}$ and the commutator expression $[y_1, z_1] \cdots [y_{\genus}, z_\genus]$. This implies that the sequence $(x_n)_{n\in\NN}$ fulfills conditions (a) and (b), where we set
\begin{equation}\label{eq:oreno_x}
  (x_n)_{n\in \NN}=([y_1, z_1^{n}] \cdots [y_\genus, z_\genus^{n}])_{n\in \NN}.
\end{equation}
Thus we obtain Theorem \ref{thm:main}.

In Section~\ref{sec:thm3}, we prove Theorem \ref{thm:main3}.
Let $\Gps$ be the fundamental group of a hyperbolic $3$-manifold that fibers over the circle with fiber $\Sigma_{\genus}$.
We note that there exists an inclusion $\iota \colon \G \to \Gps$.
We use a universal circle representation $\overline{\rho} \colon \Gps \to \h$ whose restriction to $\iota(\G)\simeq \G$ is a Fuchsian representation $\rho \colon \G \to \h$.
Let $\mu \colon \Nps \to \RR$ be the $\Gps$-invariant homogeneous quasimorphism defined via the representation $\overline{\rho}$. Take the sequence $(x_n)_{n\in \NN}$ in $[\G,\G']$ defined by \eqref{eq:oreno_x}.
By comparing the restriction $\mu|_{\G'}$ with $q'^* \mu_{\rho_{\genus}}$ (Lemma \ref{lem:comparison}), we can apply the auxiliary lemma: thus we conclude that the sequence $(\iota(x_n))_{n\in \NN}$ fulfills (a) and (b). This proves Theorem~\ref{thm:main3}.

In Section~\ref{sec:remarks}, we make concluding remarks. In the proof of Theorem \ref{thm:main3}, we use the most basic case of the universal circle representation.
In Subsection~\ref{sec:taut_foliation}, we remark a relation between $\Gps$-invariant homogeneous quasimorphisms on $\Nps$ and the universal circle representations of taut foliations on the hyperbolic mapping torus.
More precisely, we prove
that for every non-zero element of $\W(\Gps, \Nps)$, there exist taut foliations on the hyperbolic mapping torus such that it is written as a linear combination of quasimorphisms induced from the universal circle representation of the taut foliations (Proposition~\ref{prop:taut_qm}). In Subsection~\ref{subsec:overdrive}, we state our \emph{overflow argument} (Lemma~\ref{overflow}), which provides a useful sufficient condition to apply the auxiliary lemma (Lemma~\ref{lem:hojo}).
In Subsection \ref{sec:rigid}, we show that the map
\[
  \Phi \colon \Hom(G,\h)_{\basic} \to \QQQ(N)^G / \HHH^1(N;\RR)^G
\]
is \textit{rigid} under semi-conjugacy (Proposition~\ref{prop:rigidity}).

\section{Preliminaries}
Throughout the present paper, we use the following symbol: for  every  $r, s \in \RR$ and $D \geq 0$, we write $r \underset{D}{\sim} s$ to mean $|r-s| \leq D$.

\subsection{Group cohomology and bounded cohomology}\label{subsec:cohom}
For a group $G$ and $n \geq 0$, let $C_n(G)$ be the free $\ZZ[G]$-module on $G^n$ and set $C_{-1}(G) = 0$.
Let $\partial \colon C_n(G) \to C_{n-1}(G)$ be the map defined by
\[
  \partial (g_1, \ldots, g_n) = (g_2, \ldots, g_n) + \sum_{i = 1}^{n-1}(-1)^i (g_1, \ldots, g_i g_{i+1}, \ldots, g_n) + (-1)^n(g_1, \ldots, g_{n-1}).
\]
The homology $\HHH_*(G)$ of the chain complex $(C_*(G),\partial)$ is called the \emph{group homology of} $G$.

Let $\A = \ZZ$ or $\RR$.
The \emph{group cohomology $\HHH^*(G;\A)$ with coefficients in $\A$} is the homology of the dual complex of $(C_*(G),\partial)$.
An explicit cochain complex $(C^*(G;\A);\delta)$ is given by
\[
  C^n(G;\A) = \{ c \colon G^n \to \A \}
\]
and
\[
  \delta c(g_1, \ldots, g_{n+1}) = c(g_2, \ldots, g_{n+1}) + \sum_{i=1}^{n}(-1)^{i}c(g_1, \ldots, g_ig_{i+1}, \ldots, g_{n+1}) + (-1)^{n+1} c(g_1, \ldots, g_n)
\]
for $c \in C^n(G;\A)$ and $g_i \in G$.
The subset $C_b^n(G;\A)$ of all bounded functions defines a subcomplex $(C_b^*(G;\A),\delta)$.
The cohomology $\HHH_b^*(G;\A)$ of this subcomplex is called the \emph{bounded cohomology of $G$ with coefficients in $\A$}.
We note that $\HHH_b^n(G;\RR) = 0$ for all $n \geq 1$ provided that $G$ is amenable (see \cite[Theorem 3.6]{Fr}).
The inclusion $C_b^*(G;\A) \to C^*(G;\A)$ induces a homomorphism $c_G\colon \HHH_b^*(G;\A) \to \HHH^*(G;\A)$ called the \emph{comparison map}.
If $G$ is Gromov-hyperbolic, the comparison map $c_G \colon \HHH_b^*(G;\RR) \to \HHH^*(G;\RR)$ is surjective (\cite{MR919829}).

\begin{remark}\label{rem:normalized}
  A group cocycle $c \in C^n(G;\A)$ is said to be \emph{normalized} if
  \[
    c(g_1, \ldots, g_n) = 0
  \]
  whenever $g_i = 1_G$ for some $i$.
  It is known that every element of $\HHH^n(G;\A)$ (resp. $\HHH_b^n(G;\A)$) can be represented by a normalized cocycle (resp. normalized bounded cocycle).
  (\cite[Section 6]{MR0349792}; see also \cite[Proposition 2.1]{MR4087570}.)
\end{remark}

We note
that a one-cochain $c\in C^1(G;\A)$ is a one-cocycle if and only if it is a homomorphism from $G$ to $A$ by definition,
and hence the first cohomology group $\HHH^1(G;\A)$ is isomorphic to the vector space of $\A$-valued homomorphisms $\Hom(G;\A)$.
The second cohomology group $\HHH^*(G;\A)$ classifies the central $\A$-extensions of $G$ up to isomorphism, that is,
\begin{align}\label{2coh_central_ext}
  \HHH^2(G;\A) \cong \{ \text{central $\A$-extensions of $G$} \}/\{ \text{split extensions}\}
\end{align}
(see \cite[(3.12)Theorem]{brown82} for example).
For a central $\A$-extension $0 \to \A \to E \to G \to 1$, the corresponding cohomology class under \eqref{2coh_central_ext} is called the \emph{Euler class of the central extension $E$} and denoted by $e(E)$.
It is known that the Euler class $e(E)$ is an obstruction to the existence of section homomorphisms $G \to E$.
In fact, the following holds:
\begin{lemma}[see {\cite[Lemma 2.4]{Fr}}]\label{lem:lift_obstruction}
  Let $0 \to \A \to E \xrightarrow{p} G \to 1$ be a central extension, $K$ a group, and $\varphi \colon K \to G$ a group homomorphism.
  Then the pullback $\varphi^*e(E)$ is equal to zero if and only if there exists a homomorphism $\psi \colon K \to E$ such that $p \circ \psi = \varphi$.
\end{lemma}

A function $f \colon N \to \A$ is said to be \emph{$G$-invariant} if $f(gwg^{-1}) = f(w)$ for every $g \in G$ and $w \in N$.
Let $\HHH^1(N;\A)^G$ denote the vector space of  $G$-invariant $A$-valued homomorphisms.
For an exact sequence of groups $1 \to N \xrightarrow{i} G \xrightarrow{p} \GG \to 1$, the following \emph{five-term exact sequence} holds:
\begin{align}\label{five-term}
  0 \to \HHH^1(\GG;\RR) \xrightarrow{p^*} \HHH^1(G;\RR) \xrightarrow{p^*} \HHH^1(N;\RR)^G \to \HHH^2(\GG;\RR) \xrightarrow{p^*} \HHH^2(G;\RR).
\end{align}

There also exists a five-term exact sequence of bounded cohomology (see \cite[Theorem 12.0.2]{MR1840942}):
\begin{align}\label{bdd_five_term}
  0 \to \HHH_b^2(\GG;\RR) \xrightarrow{p^*} \HHH_b^2(G;\RR) \xrightarrow{i^*} \HHH_b^2(N;\RR)^G \to \HHH_b^3(\GG;\RR) \xrightarrow{p^*} \HHH_b^3(G;\RR).
\end{align}
Here $\HHH_b^2(N;\RR)^G$ is the invariant part of $G$-action on $\HHH_b^2(N;\RR)$ induced from the conjugation $G$-action on the bounded cochain group $C_b^2(N;\RR)$.

\subsection{Invariant quasimorphisms and the Bavard duality theorem for stable mixed commutator lengths}\label{subsec:qm}

A real-valued function $\mu \colon G \to \RR$ on a group $G$ is called a \emph{homogeneous quasimorphism} if there exists a non-negative real number $D$ such that
\[
\mu(gh)\underset{D}{\sim} \mu(g)+\mu(h)
\]
for every $g,h\in G$
and if it is a homomorphism on every cyclic subgroup of $G$.
The minimal value $D(\mu)$ of such $D$ is called the \emph{defect of} $\mu$: $D(\mu)=\sup_{g,h \in G} |\mu(gh) - \mu(g) - \mu(h)|$.
Every homogeneous quasimorphism $\mu \colon G \to \RR$ is $G$-invariant (recall the definition of $G$-invariance from Subsection~\ref{subsec:cohom}), and hence satisfies
\begin{align}\label{comm_defect}
  |\mu([g, h])| \leq D(\mu)
\end{align}
for every $g,h \in G$ (see \cite[Section 2.3.3]{Cal} for example).
Let $\QQQ(G)$ denote the vector space of homogeneous quasimorphisms on $G$.
Clearly $\HHH^1(G;\RR)$ is a subspace of $\QQQ(G)$.

As we mentioned in the introduction, Bavard established in \cite{bavard91} the relation between homogeneous quasimorphisms on $G$ and the stable commutator length $\scl_{G}$, which is called the \emph{Bavard duality theorem}.
\begin{theorem}[{\cite{bavard91}}]\label{thm:Bavard}
  Let $G$ be a group and $y \in [G,G]$.
  Then the following holds:
  \[
    \scl_{G}(y) = \sup_{[\mu] \in \QQQ(G)/\HHH^1(G;\RR)} \frac{|\mu(y)|}{2D(\mu)}.
  \]
\end{theorem}
Note that the right-hand side of the equality in Theorem \ref{thm:Bavard} is regarded as zero when $\QQQ(G) = \HHH^1(G;\RR)$.

Let $G$ be a group, $N$ a normal subgroup of $G$,
and $\QQQ(N)^G$ the vector space of $G$-invariant homogeneous quasimorphisms on $N$.
Note that, by the $G$-invariance of every $\mu$ in $\QQQ(G)$, the restriction $\mu|_{N}$ is contained in $\QQQ(N)^G$.
Kawasaki--Kimura--Matsushita--Mimura proved the following Bavard duality theorem for stable mixed commutator lengths, which connects  $\scl_{G,N}$ and $G$-invariant homogeneous quasimorphisms on $N$.
\begin{theorem}[{\cite[Theorem~1.2]{MR4452430}}]\label{thm:mixedBavard}
  Let $G$ be a group, $N$ a normal subgroup, and $x \in [G,N]$.
  Then the following equality
  \[
    \scl_{G,N}(x) = \sup_{[\mu] \in \QQQ(N)^G/\HHH^1(N;\RR)^G} \frac{|\mu(x)|}{2D(\mu)}.
  \]
holds true.
\end{theorem}

\subsection{The group of orientation preserving homeomorphisms of the circle}\label{subsec:poincare}

Recall from Subsection~\ref{subsec:outline} that we have set $\h = \Homeo_+(S^1)$, the group of orientation preserving homeomorphisms of the circle $S^1$.
We regard $S^1$ as $\RR/\ZZ$ throughout this paper.
For every $r \in \RR$, let $T_r \colon \RR \to \RR$ be the homeomorphism defined by $T_r(x) = x + r$ for $x \in \RR$.
We set
\[
  \th = \{ \tf \colon \RR \to \RR \mid \tf \circ T_1 = T_1 \circ \tf \}.
\]
Let $p \colon \th \to \h$ be the canonical projection.
This projection gives rise to a central $\ZZ$-extension
\begin{align}\label{univ_ext}
  0 \to \ZZ \to \th \xrightarrow{\pi} \h \to 1.
\end{align}

Eisenbud, Hirsch, and Neumann \cite{MR656217} completely determined the commutator lengths of the elements of $\th$ as follows:
For every $\tf \in \th$, we set
\[
  \underline{m}(\tf) = \min_{x \in \RR} (\tf(x) - x) \ \text{ and } \ \overline{m}(\tf) = \max_{x \in \RR} (\tf(x) - x).
\]
\begin{theorem}[{\cite{MR656217}}]\label{thm:ehn}
  Let $\tf$ be an element of $\th$. Let $n\geq 1$ be an integer.
Then, the following two conditions on $\tf$ are equivalent:
\begin{enumerate}[$(1)$]
  \item $\cl_{\th}(\tf)\leq n$.
  \item $\underline{m}(\tf) < 2n - 1$ and $\overline{m}(\tf) > 1 - 2n$.
\end{enumerate}

In particular, $\tf$ may be expressed as a single commutator of $\th$ if both of the inequalities $\underline{m}(\tf)<1$ and $\overline{m}(\tf) >-1$ hold.
\end{theorem}

Poincar\'{e} \cite{poincare81} introduced a homogeneous quasimorphism $\tau \in \QQQ(\th)$ called the \emph{translation number}.
The translation number $\tau \in \QQQ(\th)$ is defined by
\[
  \tau(\tf) = \lim_{n \to \infty} \frac{\tf^n(0)}{n}.
\]
This limit exists and defines a homogeneous quasimorphism with defect $D(\tau) = 1$ (see \cite[Lemma 2.40 and Proposition 2.92]{Cal}).

In \cite{MR848896}, Matsumoto introduced the \emph{canonical Euler cocycle} $\chi \in C^2(\h;\RR)$.
The cocycle $\chi$ is defined by
\begin{align}\label{can_eu_def}
  \chi(f,g) = \tau(\tf\widetilde{g}) - \tau(\tf) - \tau(\widetilde{g}),
\end{align}
where $\tf$ and $\widetilde{g}$ in $\th$ are arbitrary lifts of $f$ and $g$ in $\h$, respectively.
Since $\tau$ is a homogeneous quasimorphism, the right-hand side of \eqref{can_eu_def} also defines a \emph{bounded} cocycle $\chi_b \in C_b^2(\h;\RR)$, which is called the \emph{canonical bounded Euler cocycle}.

It is known that the cohomology class $[\chi] \in \HHH^2(\h;\RR)$ corresponds to the Euler class $e_{\ZZ} = e(\th) \in \HHH^2(\h;\ZZ)$ of central extension \eqref{univ_ext} under the change of coefficients homomorphism $\HHH^2(\h;\ZZ) \to \HHH^2(\h;\RR)$ (see \cite[Proposition 10.26]{Fr} for example).
We set $e_{\RR} = [\chi] \in \HHH^2(\h;\RR)$ and call it the \emph{real Euler class of $\h$}.

Now we construct a central $\RR$-extension $\h$ which corresponds to the \emph{real} Euler class of $\h$.
The group $\thr$ is defined as a quotient $(\th \times \RR)/\sim$ with the equivalence relation $(\tf, r) \sim (\tf \circ T_1, r-1)$ for $(\tf,r) \in \h \times \RR$.
The multiplication in $\thr$ is induced from $(\tf, r) \cdot (\widetilde{g}, s) = (\tf\widetilde{g}, r + s)$.
By abuse of notation, $(\tf, r)$ also denotes the element of $\thr$ represented by $(\tf, r) \in \th \times \RR$.
This group $\thr$ gives rise to a central $\RR$-extension
\begin{align}\label{cent_ext_r}
  0 \to \RR \xrightarrow{i} \thr \xrightarrow{\pi} \h \to 1,
\end{align}
where $i(r) = (\id_{\RR}, r) \in \thr$ and $\pi((\tf, s)) = \pi(\tf)$ for every $r \in \RR$ and $(\tf, s) \in \thr$.
It is verified that central $\RR$-extension \eqref{cent_ext_r} corresponds to the real Euler class $e_{\RR} \in \HHH^2(\h;\RR)$ (see \cite[Remark 1]{Moriyoshi16} for example).

Let us define a map $\taur \colon \thr \to \RR$ by
\begin{align}\label{real_transl_number}
  \taur((\tf, r)) = \tau(\tf) + r.
\end{align}
This $\taur$ is well-defined since $\tau(\tf \circ T_1) = \tau(\tf) + 1$.
Moreover, this is a homogeneous quasimorphism since
\[
  \taur((\tf, r)^n) = \taur((\tf^n, nr)) = \tau(\tf^n) + rn = n\cdot \taur((\tf, r))
\]
and
\[
  |\taur((\tf, r)\cdot(\tg, s)) - \taur((\tf, r)) - \taur((\tg, s))|
  = |\tau(\tf \tg) - \tau(\tf) - \tau(\tf)| \leq D(\tau) = 1
\]
for every $(\tf, r), (\tg, s) \in \thr$ and for every integer $n$.
In particular, we have $D(\taur) = D(\tau) = 1$.
Note that for every $f, g \in \h$, the equality
\begin{align}\label{can_euler_real_transl}
  \chi(f,g) = \chi_b(f, g) = \taur((\tf, r)\cdot(\tg, s)) - \taur((\tf, r)) - \taur((\tg, s))
\end{align}
holds, where $(\tf, r), (\tg, s) \in \thr$ are lifts of $f, g$, respectively.
Note also that every homogeneous quasimorphism $\mu$ satisfies
\begin{align}\label{hqm_commutes}
  \mu(xy) = \mu(x) + \mu(y)
\end{align}
for elements $x$ and $y$ with $xy = yx$, and so does $\taur$.

\section{Construction of invariant homogeneous quasimorphisms}\label{sec:construction}

Let $G$ be a group and $N$ a normal subgroup.
In this section, we construct a $G$-invariant homogeneous quasimorphism on $N$. Set $\Gamma=G/N$.

\begin{lemma}\label{lem:construction}
  Let $G$ be a group, $N$ a normal subgroup, $i \colon N \to G$ the inclusion, and $p \colon G \to \Gamma$ the projection. Set $\Gamma=G/N$.
  Let $c_b \in C_b^2(G;\RR)$ be a bounded two-cocycle satisfying the following two conditions.
  \begin{enumerate}[$(1)$]
    \item There exists a homogeneous quasimorphism $\mu \in \QQQ(N)$ such that $i^* c_b = \delta \mu$.
    \item There exist a normalized cocycle $A \in C^2(\Gamma;\RR)$ and a cochain $u \in C^1(G;\RR)$ such that $c_b - p^*A = \delta u$.
  \end{enumerate}
  Then the restriction $u|_{N} \colon N \to \RR$ is a $G$-invariant homogeneous quasimorphism.
  This quasimorphism $u|_N$ is written as $u|_N = \mu + h$ for some homomorphism $h \in \HHH^1(N;\RR)$.
\end{lemma}

\begin{proof}
  By condition~(1), we have
  \[
    \delta (u|_N) = \delta (i^* u) = i^*(\delta u) = i^* c_b = \delta \mu.
  \]
  Hence there exists a one-cocycle $h \in C^1(N;\RR)$ such that $u|_N = \mu + h$.
  Since every one-cocycle is a homomorphism, $u|_N$ is a homogeneous quasimorphism on $N$.

  Now we show the $G$-invariance of $u|_N$.
  For every $g \in G$ and $w \in N$, we consider group two-chains $(g^{-1}, wg)$, $(w,g)$, and $(g^{-1}, g)$ in $C_2(G)$.
  Substituting them for $c_b - p^*A = \delta u$ by condition~(2), we obtain
  \begin{align}\label{eq1}
    u(g^{-1}wg) &= u(wg) + u(g^{-1}) - c_b(g^{-1}, wg) + p^*A(g^{-1}, wg), \nonumber \\
    u(wg) &= u(g) + u(w) - c_b(w, g) + p^*A(w, g),  \\
    0 = u(1_G) &= u(g) + u(g^{-1}) - c_b(g^{-1}, g) + p^*A(g^{-1}, g) \nonumber.
  \end{align}
  Since $N = \Ker (p)$ and $A$ is normalized, we have
  \begin{align}\label{eq2}
    p^*A(g^{-1}, wg) = p^*A(g^{-1}, g) \ \text{ and } \ p^*A(w, g) = 0.
  \end{align}
  Equalities \eqref{eq1} and \eqref{eq2} imply
  \begin{align*}
    u(g^{-1}wg) - u(w) = c_b(g^{-1}, g) - c_b(w, g) - c_b(g^{-1}, wg).
  \end{align*}
  Since $w \in N$ is arbitrary and $u|_N$ is homogeneous, we have
  \begin{align}\label{eq3}
    n|u(g^{-1}wg) - u(w)| = |u(g^{-1} w^n g) - u(w^n)| = |c_b(g^{-1}, g) - c_b(w^n, g) - c_b(g^{-1}, w^ng)|
  \end{align}
  for every $n \in \ZZ$.
  Since $c_b$ is a bounded cochain, the very right-hand side of \eqref{eq3} is uniformly bounded.
  This implies that $u(g^{-1}wg) - u(w) = 0$.
\end{proof}

\begin{remark}\label{rem:unique_as_class}
  As an element of $\QQQ(N)^G/(\HHH^1(N;\RR)^G + i^*\QQQ(G))$, the resulting $G$-invariant homogeneous quasimorphism $u|_{N}$ in Lemma \ref{lem:construction} depends only on the class $c_G([c_b]) \in \HHH^2(G;\RR)$.
  To see this, let $d_b \in C_b^2(G;\RR)$ be a bounded two-cocycle satisfying the equality $c_G([d_b]) = c_G([c_b]) \in \HHH^2(G;\RR)$.
  We take $\nu \in \QQQ(N), B \in C^2(\Gamma;\RR)$, and $v \in C^1(G;\RR)$ satisfying conditions (i) and (ii) in Lemma \ref{lem:construction}.
  Since $c_G([c_b - d_b]) = 0$, there exist a homogeneous quasimorphism $f \in \QQQ(G)$ and a bounded cochain $z \in C_b^1(G;\RR)$ such that $\delta f + \delta z = c_b - d_b$.
  Then we have
  \begin{align*}
    \delta(u|_N - v|_N) = i^* (\delta u - \delta v) = i^*(c_b - d_b) = \delta(i^*f + i^*z).
  \end{align*}
  Hence the difference $(u|_N - v|_N) - (i^*f + i^*z)$ is a homomorphism from $N$ to $\RR$.
  In particular, the boundedness of $z$ implies that $i^*z$ is the zero map.
  Therefore, $u|_N - v|_N$ belongs to $\HHH^1(N;\RR)^G + i^*\QQQ(G)$.
\end{remark}

\begin{lemma}\label{lem:non-ext}
  Under the setting in Lemma $\ref{lem:construction}$, we further assume that $\HHH_b^2(\Gamma;\RR) = 0$ and the class $c_G([c_b]) \in \HHH^2(G;\RR)$ is non-zero.
  Then, the resulting $G$-invariant homogeneous quasimorphism $u|_{N}$ is non-zero in $\W(G,N)$.
\end{lemma}

\begin{proof}
  Recall that $\W(G,N) = \QQQ(N)^G/(\HHH^1(N)^G + i^*\QQQ(G))$.
  Assume for contradiction that $u|_{N}$ is equal to zero in $\W(G,N)$.
  Then, there exist an element $h' \in \HHH^1(N;\RR)^G$ and an element $\mu' \in \QQQ(G)$ such that $u|_{N} = h' + i^*\mu'$,
  and hence we have
  \begin{align*}
    i^*(\delta \mu') = i^*(\delta u) = i^*c_b.
  \end{align*}
  Since $\HHH_b^2(\Gamma;\RR) = 0$, the homomorphism $i^* \colon \HHH_b^2(G;\RR) \to \HHH_b^2(N;\RR)$ is injective by \eqref{bdd_five_term}.
  Hence the bounded cohomology class $[c_b] \in \HHH_b^2(G;\RR)$ coincides with $[\delta \mu]$.
  This implies that $c_G([\delta \mu]) = 0$, which contradicts the assumption.
\end{proof}

Lemma \ref{lem:construction} greatly works when the group $G$ has a representation $\rho \colon G \to \h (=\Homeo_+(S^1))$, and in this case the resulting $G$-invariant homogeneous quasimorphism has a nice formula.

Recall that the real Euler class $e_{\RR} \in \HHH^2(\h;\RR)$ is the Euler class of extension \eqref{cent_ext_r}.
We now construct $G$-invariant homogeneous quasimorphisms $\nu_{\rho, A, u}$ in the introduction.
\begin{construction}\label{construction}
  Let $G$ be a group and $N$ its normal subgroup. Set the group quotient $\Gamma=G/N$ and the projection $p \colon G \to \Gamma$.  Let $\rho \colon G \to \h$ be a homomorphism.
  Assume that the pullback $\rho^*e_{\RR}$ is contained in the image of the map $p^* \colon \HHH^2(\Gamma;\RR) \to \HHH^2(G;\RR)$.

  Since $\rho^*e_{\RR}$ is contained in the image of $p^* \colon \HHH^2(\Gamma;\RR) \to \HHH^2(G;\RR)$, we have that $(\rho \circ i)^*e_{\RR} = 0$.
  Hence, by Lemma \ref{lem:lift_obstruction}, there exists a homomorphism $\trho \colon N \to \thr$ such that the following commutes:
  \begin{align}\label{comm_diag_gen}
    \xymatrix{
    N \ar[r]^-{\trho} \ar[d]^-{i} & \thr \ar[d]^-{\pi} \\
    G \ar[r]^-{\rho} & \h.
    }
  \end{align}
  Recall that $\chi_b \in C_b^2(\h;\RR)$ denotes the canonical bounded Euler cocycle.
  The cocycle $\rho^*\chi_b$ satisfies condition~(1) in Lemma \ref{lem:construction} because
  \begin{align*}
    i^* \rho^* \chi_b = \trho^* \pi^* \chi_b = \trho^* (-\delta \taur) = \delta (-\trho^* \taur)
  \end{align*}
  by \eqref{can_euler_real_transl} and \eqref{comm_diag_gen}.
  Remark \ref{rem:normalized}, together with the assumption, implies that condition~(2) in Lemma \ref{lem:construction} holds, that is, there exist a normalized cocycle $A \in C^2(\Gamma;\RR)$ and a cochain $u \in C^1(G;\RR)$ such that $\rho^* \chi_b - p^*A = \delta u$.
  We set $\nu_{\rho, A, u} = u|_N$.

  By Lemma \ref{lem:construction}, $\nu_{\rho, A, u}$ is a $G$-invariant homogeneous quasimorphism on $N$ and $\nu_{\rho, A, u} = -\trho^*\taur + h$ for some homomorphism $h \in \HHH^1(N;\RR)$.
  Moreover, since $D(\taur) = 1$, we have $D(\nu_{\rho, A, u}) \leq 1$.
\end{construction}

\begin{remark}\label{rem:indep}
  The homogeneous quasimorphism $\nu_{\rho,A,u}$ in Construction \ref{construction} depends on $A$ and $u$.
  However, if $N$ is contained in $G'$, then for two choices $(A_1,u_1)$ and $(A_2,u_2)$ of $A$ and $u$ the difference $\nu_{\rho,A_1,u_1}-\nu_{\rho,A_2,u_2}$ lies in $\HHH^1(N;\mathbb{R})^G$.
  Therefore, under this additional assumption, the class $[\mu_{\rho}]$ in $\QQQ(N)^G/\HHH^1(N;\mathbb{R})^G$ is determined by $\rho$ and does not depend on the choice of $A$ or $u$.
  In other words, we obtain a well-defined map
  \[
    \Phi \colon \Hom(G, \h)_{\basic} \to \QQQ(N)^G/\HHH^1(N;\mathbb{R})^G,
  \]
  where $\Hom(G, \h)_{\basic}$ is the space defined in Subsection~\ref{subsec:intro_inv_qm}.
  We will see in Section \ref{sec:remarks} that this map $\Phi$ is rigid under semi-conjugacy.
\end{remark}

\begin{proof}[Proof of Theorem $\ref{thm:intro_non-ext}$]
  This is immediately from Lemma \ref{lem:non-ext}.
\end{proof}

As we mentioned in the introduction, we furthermore provide the following explicit formula on the resulting quasimorphisms in terms of the Poincar\'{e} translation number.

\begin{theorem}[explicit formula]\label{thm:formula}
Let $G$ be a group and $N$ its normal subgroup. Set the group quotient $\Gamma=G/N$ and the projection $p \colon G \to \Gamma$.  Let $\rho \colon G \to \h$ be a homomorphism.
Assume that the pullback $\rho^*e_{\RR}$ is contained in the image of the map $p^* \colon \HHH^2(\Gamma;\RR) \to \HHH^2(G;\RR)$.
For $A$ and $u$ with $\rho^* \chi_b - p^*A = \delta u$, we set $\mu_{\rho} = \nu_{\rho, A, u}$.
Then the following hold true.
\begin{enumerate}[$(1)$]
  \item $\mu_{\rho}$ is a $G$-invariant homogeneous quasimorphism of defect $D(\mu_{\rho}) \leq 1$.
  \item Let $k$ be a positive integer.
        Then for every $g_1, \ldots, g_k \in G$ and $w_1, \ldots, w_k \in N$,
    \[
      \mu_{\rho}([g_1, w_1] \cdots [g_k, w_k]) = -\taur\left([\widetilde{\rho(g_1)}, \widetilde{\rho(w_1)}] \cdots [\widetilde{\rho(g_k)}, \widetilde{\rho(w_k)}]\right).
    \]
    Here $\taur \colon \th \to \RR$ is the map defined by \eqref{real_transl_number} and $\widetilde{\rho(g_i)}, \widetilde{\rho(w_i)} \in \thr$ are lifts of $\rho(g_i), \rho(w_i) \in \h$, respectively, for every $i=1,\ldots ,k$.
  \end{enumerate}
\end{theorem}


\begin{proof}
  Condition (1) is checked in Construction \ref{construction}.
  We show (2).
  For $g_1, \ldots, g_k \in G$ and $w_1, \ldots, w_k \in N$, we define a group two-chain $\sigma \in C_2(G)$ by
  \begin{align*}
    \sigma =& \sum_{i = 1}^{k} (g_i, g_i^{-1}) + \sum_{i=1}^{k} (w_i,  w_i^{-1}) + 2k(1_G, 1_G)\\
    -& \big( (g_1, w_1) + (g_1 w_1, g_1^{-1}) + (g_1 w_1 g_1^{-1}, w_1^{-1}) + ([g_1, w_1], g_2) \\
    &+ ([g_1, w_1]g_2, w_2) + ([g_1, w_1]g_2 w_2, g_2^{-1}) + \cdots \\
    &+ ([g_1, w_1]\cdots [g_{k-1}, w_{k-1}]g_k w_{k}, g_k^{-1}) + ([g_1, w_{1}]\cdots [g_{k-1}, w_{k-1}]g_k w_{k} g_k^{-1}, w_{k}^{-1}) \big).
  \end{align*}
  This cochain satisfies $\partial \sigma = [g_1, w_{1}] \cdots [g_k, w_{k}]$.
  Hence we have
  \begin{align*}
   \nu_{\rho, A, u}([g_1, w_{1}] \cdots [g_k, w_{k}]) = u(\partial \sigma) = \delta u(\sigma) = (\rho^*\chi_b - p^*A)(\sigma) = \chi_b(\rho_*\sigma) - A(p_*\sigma).
  \end{align*}
  Since $A$ is a normalized cochain, $A(p_*\sigma) = 0$ holds.
  Let $\widetilde{\rho(g_i)}, \widetilde{\rho(w_{i})} \in \thr$ be lifts of $\rho(g_i), \rho( w_{i}) \in \h$, respectively, for every $i=1,\ldots ,k$.
  Let $\widetilde{\sigma} \in C_2(\thr)$ be the two-chain defined by replacing $g_i, w_{i}$, and $1_G$ in each term of $\sigma$ by $\widetilde{\rho(g_i)}, \widetilde{\rho(w_{i})}$, and $1_{\thr}$, respectively, for every $i=1,\ldots ,k$.
  Then $\pi_* \widetilde{\sigma} = \rho_*\sigma$.
  Hence, by $\chi_b = -\pi^* (\delta \taur)$ and $\delta \taur(\tf, \tf^{-1}) = 0$ for every $\tf \in \thr$, we have
  \begin{align*}
    \chi_b(\rho_* \sigma) = -\delta \taur(\widetilde{\sigma}) = -\taur\left([\widetilde{\rho(g_1)}, \widetilde{\rho(w_{1})}] \cdots [\widetilde{\rho(g_k)}, \widetilde{\rho(w_{k})}]\right).
  \end{align*}
This completes the proof of (2).
\end{proof}


\section{The auxiliary lemma}\label{subsec:overflow}

Recall our symbol \eqref{eq:conj} for group conjugation. We note that the equality
\begin{align}\label{ore_relation}
  [g, h^n] = [g, h] \cdot {}^{h}[g,h] \cdot \cdots \cdot {}^{h^{n-2}}[g, h] \cdot {}^{h^{n-1}}[g, h].
\end{align}
holds for every $g, h \in G$ and for every positive integer $n > 0$.


%

\begin{lemma}\label{lem:[G,N]}
  Let $G$ be a group, $N$ a normal subgroup, and $\mu \in \QQQ(N)^G$.
Let $k$ be a positive integer and $y_1, \ldots, y_k, z_1, \ldots, z_k \in G$. Assume that the following two conditions are satisfied:
  \begin{enumerate}[$(i)$]
    \item $[y_i, z_i] \in N$ for every $i = 1, \ldots, k$;
    \item $[y_1, z_1] \cdots [y_k, z_k] \in [G, N]$.
  \end{enumerate}
Then for every $n\in \NN$,
\[
  [y_1, z_1^n] \cdots [y_k, z_k^n]\in [G,N] \quad \textrm{and}\quad [y_1^n, z_1] \cdots [y_k^n, z_k]\in [G,N]
\]
hold.
\end{lemma}

\begin{proof}
We first show $[y_1, z_1^n] \cdots [y_k, z_k^n]\in [G,N]$. 
  We set $\overline{g} = g[G,N] \in G/[G,N]$ for every $g \in G$.
  Note that, for every $g \in G$ and $w \in N$, we have $\overline{g} \cdot \overline{w} = \overline{w} \cdot \overline{g}$, and hence $\overline{g} \cdot \overline{w} \cdot \overline{g}^{-1} = \overline{w}$.
  Condition (i), together with \eqref{ore_relation}, implies that
  \begin{align*}
    [\overline{y_i}, \overline{z_i}^n] = [\overline{y_i}, \overline{z_i}] \cdot {}^{\overline{z_i}} [\overline{y_i}, \overline{z_i}] \cdot
    \cdots \cdot  {}^{\overline{z_i}^{n-2}} [\overline{y_i}, \overline{z_i}] \cdot  {}^{\overline{z_i}^{n-1}} [\overline{y_i}, \overline{z_i}]
    = [\overline{y_i}, \overline{z_i}]^n.
  \end{align*}
  Since $[\overline{y_i}, \overline{z_i}]$'s commute with each other by (i), we obtain   \begin{align*}
   [\overline{y_1}, \overline{z_1}^n]\cdots [\overline{y_k}, \overline{z_k}^n] &= [\overline{y_1}, \overline{z_1}]^n \cdots [\overline{y_k}, \overline{z_k}]^n = ([\overline{y_1}, \overline{z_1}] \cdots [\overline{y_k}, \overline{z_k}])^n  \\
    &= (\overline{[y_1, z_1]\cdots [y_k, z_k]})^n = 1_{G/[G,N]}.
  \end{align*}
  Here the last equality comes from (ii). We can prove the latter in a similar manner, using
\[ [y_i^n, z_i] = {}^{z_i^{n-1}}[y_i, z_i] \cdot {}^{z_i^{n-2}}[y_i,z_i] \cdot \cdots \cdot [y_i, z_i]. \qedhere\]
\end{proof}

The following auxiliary lemma is one key to the proofs of Theorems~\ref{thm:main} and \ref{thm:main3}.
\begin{lemma}[the auxiliary lemma]\label{lem:hojo}
  Let $G$ be a group, $N$ a normal subgroup, and $\mu \in \QQQ(N)^G$.
Let $k$ be a positive integer, and $y_1, \ldots, y_k, z_1, \ldots, z_k \in G$. Assume that the following three conditions are satisfied:
  \begin{enumerate}[$(i)$]
    \item same as condition~$(i)$ in Lemma~\textup{\ref{lem:[G,N]}};
    \item same as condition~$(ii)$ in Lemma~\textup{\ref{lem:[G,N]}};
    \item $\lim\limits_{n \to \infty}|\mu([y_1, z_1^n] \cdots [y_k, z_k^n])| = \infty$.
  \end{enumerate}
Then, for the sequence
\[
  (x_n)_{n\in \NN}=([y_1, z_1^n] \cdots [y_k, z_k^n])_{n\in \NN},
\]
we have
\[
\sup_{n\in \NN}\scl_G(x_n)<\infty \quad \textrm{but}\quad \lim_{n\to \infty}\scl_{G,N}(x_n)=\infty.
\]
In particular,
$\scl_{G}$ and $\scl_{G,N}$ are \emph{not} equivalent.
\end{lemma}

\begin{proof}
First, note that by (i) and (ii),  Lemma~\ref{lem:[G,N]} shows that $(x_n)_{n\in \NN}$ is a sequence in $[G,N]$. We have
\[
\sup_{n\in \NN}\scl_{G}(x_n)\leq \sup_{n\in \NN}\cl_G(x_n)\leq k<\infty.
\]
Contrastingly, by condition  (iii), Theorem \ref{thm:mixedBavard} implies that
  \[
    \lim_{n \to \infty}\scl_{G,N}(x_n) = \infty,
  \]
as desired (here we only use the estimate of $\scl_{G,N}$ by Theorem \ref{thm:mixedBavard} from below: it is the easy direction).
\end{proof}


In Lemma~\ref{overflow}, we provide a sufficient condition to apply Lemma~\ref{lem:hojo} that seems useful.

We next provide an example of commutators which satisfies conditions (i) and (ii) in Lemma~\ref{lem:[G,N]}.


\begin{definition}\label{def:g_x}
  We define words $g_i$ and $w_i$ on the alphabet $\{a,b\}$ as follows:
  \[
    g_i =
    \begin{cases}
      bab^{-1} & i = 1 \\
      {}^{ba^{2-i}b^{-1}}a^{i - 1} & i > 1
    \end{cases}
    \ \text{ and } \
    w_i =
    \begin{cases}
      [b,a] & i = 1 \\
      {}^{ba^{2-i}b^{-1}} [b,a^{-1}] & i > 1.
    \end{cases}
  \]
\end{definition}

By induction, we obtain the following:

\begin{lemma}\label{lem:relation}
  In the free group $F_2 = \langle a, b \mid - \rangle$ of rank $2$, the equality
  \[
    [g_1, w_1] \cdots [g_{\genus-1}, w_{\genus-1}] = ba^2b^{-1}a^{-2}[a,b]^\genus a^{2-\genus}ba^{\genus-2}b^{-1}
  \]
  holds for every $\genus \geq 2$.
\end{lemma}

We set $y = {}^{ba^2} b^{-1}$ and $z = {}^{ba^2} a^{-\genus}$.
Since $[a, b]^{\genus}$ is trivial in the group $\R = \langle a, b \mid [a,b]^{\genus} \rangle$, we have the following:

\begin{corollary}\label{cor:relation_in_one_relator_grp}
  In $\R$, the equality
  \[
    [g_1, w_1] \cdots [g_{\genus-1}, w_{\genus-1}] = [y,z]
  \]
  holds for every $\genus \geq 2$.
  In particular, $[y,z^n]\in [\R, \R']$ for every positive integer $n$.
\end{corollary}

\section{One-relator groups with torsion}\label{sec:thm2}

In this section, we show the following:
\begin{theorem}\label{thm:main2}
For $\genus \geq 2$,  $\scl_{\R}$ and $\scl_{\R, \R'}$ are \emph{not} equivalent.
\end{theorem}

Let $p\colon \R \to \R/\R' \simeq \ZZ^{2}$ be the projection.
\begin{lemma}\label{lem:one_relator_from_quotient}
  The map $p^* \colon \HHH^2(\R/\R';\RR) \to \HHH^2(\R;\RR)$ is an isomorphism.
\end{lemma}

\begin{proof}
  Let $\llangle {[a, b]^\genus} \rrangle$ be the normal closure of $[a, b]^\genus$ in $F_2$.
  By five-term exact sequence \eqref{five-term} applied to $1 \to \llangle {[a, b]^\genus} \rrangle \to F_2 \to \R \to 1$ and the triviality of $\HHH^2(F_2;\RR)$,
  the map
  \[
    \HHH^1(\llangle {[a, b]^\genus} \rrangle;\RR)^{F_2} \to \HHH^2(\R;\RR)
  \]
  is an isomorphism.
  Since an $F_2$-invariant homomorphism from $\llangle {[a, b]^\genus} \rrangle$ to $\RR$ is determined by its value on ${[a, b]^\genus}$, we have $\HHH^2(\R;\RR) \cong \HHH^1(\llangle {[a, b]^\genus} \rrangle;\RR)^{F_2} \cong \RR$.

  Five-term exact sequence \eqref{five-term} applied to $1 \to \R' \to \R \to \R/\R' \to 1$ asserts that
  \[
    \cdots \to \HHH^1(\R';\RR)^{\R} \to \HHH^2(\R/\R';\RR) \to \HHH^2(\R;\RR)
  \]
  is exact.
  Since $\HHH^1(\R';\RR)^{\R} = 0$ and $\HHH^2(\R/\R';\RR) \cong \RR$, the map $\HHH^2(\R/\R';\RR) \to \HHH^2(\R;\RR)$ is an injection from $\RR$ to $\RR$.
  This implies the lemma.
\end{proof}

For the proof of Theorem~\ref{thm:main2}, we employ the words $g_1,\ldots ,g_{\genus-1}, w_1,\ldots ,w_{\genus-1}$ on the alphabet $\{a,b\}$ defined in Definition~\ref{def:g_x}. We employ two more words $y = {}^{ba^2} b^{-1}$ and $z = {}^{ba^2} a^{-\genus}$. Corollary~\ref{cor:relation_in_one_relator_grp} states that the equality
\[
[y,z]=[g_1,w_1]\cdots [g_{\genus -1},w_{\genus-1}]
\]
holds \emph{in $\R$}.

\begin{proof}[Proof of Theorem \textup{\ref{thm:main2}}]
In this proof, we regard $y = {}^{ba^2} b^{-1}$ and $z = {}^{ba^2} a^{-\genus}$ as elements in $\R$.
  Recall that the map $T_r \colon \RR \to \RR$ is defined by $T_r(x) = x + r$ for every $r \in \RR$.
  By Theorem \ref{thm:ehn}, there exist $\tf, \tg \in \th$ such that
\[
   T_{\frac{\genus - 1}{\genus}} = [\tf, \tg]
\]
holds.
  We set $\talpha = (\tf, 0)$ and $\tbeta = (\tg, 0) \in \thr$.
  Then we have
  \[
    [\talpha, \tbeta] = ([\tf, \tg], 0) = (T_{\frac{\genus-1}{\genus}}, 0).
  \]
  We set $\alpha = \pi(\talpha)$ and $\beta = \pi(\tbeta)$ in $\h$.
  Then the map $\rho_{\genus} \colon \R \to \h$ defined by $\rho_{\genus}(a) = \alpha$ and $\rho_{\genus}(b) = \beta$ is a well-defined homomorphism.
  Indeed,
  \[
    \rho_{\genus}([a,b]^{\genus}) = [\alpha, \beta]^{\genus} = [\pi(\talpha), \pi(\tbeta)]^{\genus} = \pi([\talpha, \tbeta]^{\genus}) = \pi((T_{\genus-1}, 0)) = 1_{\h}.
  \]
By Lemma \ref{lem:one_relator_from_quotient}, we can apply Construction \ref{construction} to the group pair $(G,N)=(\R,\R')$ and to this homomorphism $\rho_{\genus}\colon \R\to \h$.
(In fact, for this pair, Construction \ref{construction} applies to an arbitrary homomorphism $\rho\colon \R\to \h$.)
Thus, the resulting invariant homogeneous quasimorphism $\mu_{\rho_{\genus}}\colon \R'\to \RR$ satisfies (1) and (2) of Theorem \ref{thm:formula}. In particular, $\mu_{\rho_{\genus}}\in \QQQ(\R')^{\R}$ and $D(\mu_{\rho_{\genus}})\leq 1$.
In what follows, we will verify that Lemma~\ref{lem:hojo} applies to the case where $k=1$, $y_1=y$, $z_1=z$ and $\mu=\mu_{\rho_{\genus}}$. Condition~(i) trivially holds since we take $N=\R'$. Condition~(ii) holds by Corollary~\ref{cor:relation_in_one_relator_grp}. Therefore, it only remains to show that condition~(iii) is satisfied, namely,
\begin{equation}\label{eq:infty_1}
 \lim_{n\to \infty}|\mu_{\rho_{\genus}}([y,z^n])|=\infty
\end{equation}
holds.



  We define elements $\widetilde{\rho_{\genus}(g_i)}$ and $\widetilde{\rho_{\genus}(w_i)}$ of $\thr$ by replacing each $a$ and $b$ in the words $g_i$ and $w_i$ in Definition \ref{def:g_x}
  with elements $\talpha$ and $\tbeta$ of $\thr$, respectively, for every $i=1,\ldots ,\genus -1$.
  We set $\try = {}^{\tbeta \talpha^2} \tbeta^{-1}$.
  These $\widetilde{\rho_{\genus}(g_i)}$, $\widetilde{\rho_{\genus}(w_i)}$, and $\try$ are lifts of $\rho_{\genus}(g_i)$, $\rho_{\genus}(w_i)$, and $\rho_{\genus}(y)$, respectively, for every $i=1,\ldots ,\genus-1$.
  By Lemma \ref{lem:relation} and the fact that $[\talpha, \tbeta]^{\genus} = (T_{\genus-1}, 0)$ is contained in the center of $\thr$, we obtain
  \begin{align}\label{eq:H}
    [\widetilde{\rho_{\genus}(g_1)}, \widetilde{\rho_{\genus}(w_1)}] \cdots [\widetilde{\rho_{\genus}(g_{\genus-1})}, \widetilde{\rho_{\genus}(w_{\genus-1})}]
    &= \tbeta \talpha^2 \tbeta^{-1} \talpha^{-2}[\talpha,\tbeta]^\genus \talpha^{2-\genus}\tbeta \talpha^{\genus-2}\tbeta^{-1} \nonumber \\
    &= \tbeta \talpha^2 \tbeta^{-1} \talpha^{-2} \talpha^{2-\genus}\tbeta \talpha^{\genus-2}\tbeta^{-1} \cdot [\talpha,\tbeta]^\genus \\
    & = {}^{\tbeta \talpha^2} [\tbeta^{-1}, \talpha^{-\genus}]\cdot [\talpha,\tbeta]^\genus \nonumber.
  \end{align}

Let $n\in \NN$. By \eqref{ore_relation} and the equality $[y,z]=[g_1,w_1]\cdots [g_{\genus -1},w_{\genus-1}]$, we have
  \begin{align*}
    [y, z^n] =& \big([g_1,w_1] \cdots [g_{\genus - 1}, w_{\genus - 1}]\big) \cdot {}^{z}\big([g_1,w_1] \cdots [g_{\genus - 1}, w_{\genus - 1}]\big) \\
    & \cdots {}^{z^{n-2}}\big([g_1,w_1] \cdots [g_{\genus - 1}, w_{\genus - 1}]\big) \cdot  {}^{z^{n-1}}\big([g_1,w_1] \cdots [g_{\genus - 1}, w_{\genus - 1}]\big).
  \end{align*}
  By \eqref{eq:H}, we obtain
  \begin{align*}
    & \big([\widetilde{\rho_{\genus}(g_1)}, \widetilde{\rho_{\genus}(w_1)}] \cdots [\widetilde{\rho_{\genus}(g_{\genus-1})}, \widetilde{\rho_{\genus}(w_{\genus-1})}]\big)
    \cdot {}^{\trrz}\big([\widetilde{\rho_{\genus}(g_1)}, \widetilde{\rho_{\genus}(w_1)}] \cdots [\widetilde{\rho_{\genus}(g_{\genus-1})}, \widetilde{\rho_{\genus}(w_{\genus-1})}]\big) \cdot \\
    & \cdots {}^{\trrz^{n-2}}\big([\widetilde{\rho_{\genus}(g_1)}, \widetilde{\rho_{\genus}(w_1)}] \cdots [\widetilde{\rho_{\genus}(g_{\genus-1})}, \widetilde{\rho_{\genus}(w_{\genus-1})}]\big)
    \cdot {}^{\trrz^{n-1}}\big([\widetilde{\rho_{\genus}(g_1)}, \widetilde{\rho_{\genus}(w_1)}] \cdots [\widetilde{\rho_{\genus}(g_{\genus-1})}, \widetilde{\rho_{\genus}(w_{\genus-1})}]\big) \\
    = & \big({}^{\tbeta \talpha^2} [\tbeta^{-1}, \talpha^{-\genus}]\big)
    \cdot {}^{\trrz}\big({}^{\tbeta \talpha^2} [\tbeta^{-1}, \talpha^{-\genus}]\big) \cdots {}^{\trrz^{n-2}}\big({}^{\tbeta \talpha^2} [\tbeta^{-1}, \talpha^{-\genus}]\big)
    \cdot {}^{\trrz^{n-1}}\big({}^{\tbeta \talpha^2} [\tbeta^{-1}, \talpha^{-\genus}]\big) \cdot \big([\talpha,\tbeta]^\genus\big)^{n} \\
    = & {}^{\tbeta \talpha^2} \big( [\tbeta^{-1}, \talpha^{-\genus}]
    \cdot {}^{\talpha^{-\genus}} [\tbeta^{-1}, \talpha^{-\genus}] \cdots {}^{(\talpha^{-\genus})^{n-2}}[\tbeta^{-1}, \talpha^{-\genus}] \cdot {}^{(\talpha^{-\genus})^{n-1}}[\tbeta^{-1}, \talpha^{-\genus}] \big)
    \cdot \big([\talpha,\tbeta]^\genus\big)^{n} \\
    = & {}^{\tbeta \talpha^2} [\tbeta^{-1}, \talpha^{-\genus n}]\cdot ([\talpha,\tbeta]^{\genus})^n.
  \end{align*}
  Here the last equality comes from \eqref{ore_relation}.
  Hence, by Theorem \ref{thm:formula} and (\ref{hqm_commutes}), we obtain
  \begin{align*}
   \mu_{\rho_{\genus}} ([y, z^n]) &= -\taur\left( {}^{\tbeta \talpha^2} [\tbeta^{-1}, \talpha^{-\genus n}]\cdot ([\talpha,\tbeta]^{\genus})^n \right) = -\taur({}^{\tbeta \talpha^2} [\tbeta^{-1}, \talpha^{-\genus n}]) - \taur(([\talpha,\tbeta]^{\genus})^n) \\
   &=-\taur({}^{\tbeta \talpha^2} [\tbeta^{-1}, \talpha^{-\genus n}]) - \tau(T_{\genus-1}^n) = -\taur({}^{\tbeta \talpha^2} [\tbeta^{-1}, \talpha^{-\genus n}]) - n(\genus-1).
 \end{align*}
By \eqref{comm_defect}, we have $|\taur({}^{\tbeta \talpha^2} [\tbeta^{-1}, \talpha^{\genus n}])| \leq D(\taur) = 1$.
  Hence we conclude that
  \begin{align}\label{single_comm_m}
    \left|\mu_{\rho_{\genus}} ([y, z^n])\right| \geq n(\genus - 1) - 1.
  \end{align}
  Then, \eqref{eq:infty_1} immediately follows from \eqref{single_comm_m}.

  Therefore, Lemma~\ref{lem:hojo} applies: the sequence $([y,z^n])_{n\in \NN}$ witnesses the non-equivalence of $\scl_{\R}$ and $\scl_{\R,\R'}$.
\end{proof}

\section{Proof of Theorem \ref{thm:main}}\label{sec:thm1}

In this section, we prove Theorem \ref{thm:main}.
Let $\G$ be the fundamental group of a closed oriented surface of genus $\genus \geq 2$.
The group $\G$ has the following presentation:
\begin{align*}
  \G = \langle a_1, \ldots, a_{\genus}, b_1, \ldots, b_{\genus} \mid [a_1, b_1] \cdots [a_\genus, b_\genus] \rangle.
\end{align*}

Let $q \colon \G \to \R$ be a map defined by
\[
  q(a_i) = a \ \text{ and } \ q(b_i) = b
\]
for each $i = 1, \ldots ,\genus$.
Since $q([a_1, b_1] \cdots [a_\genus, b_\genus]) = [a,b]^{\genus} = 1_{\R}$, the map $q$ is a well-defined homomorphism.
Let $\G'$ be the commutator subgroup and $q' \colon \G' \to \R'$ the induced map.

Let $\rho_{\genus} \colon \R \to \h$ be the homomorphism and $\mu_{\rho_{\genus}}$ the $\R$-invariant homogeneous quasimorphism on $\R'$ used in the proof of Theorem \ref{thm:main2}.
From the proof of Theorem~\ref{thm:main2}, we continue to use the representation $\rho_{\genus} \colon \R \to \h$ and the resulting quasimorphism $\mu_{\rho_{\genus}}\in \QQQ(\R')^{\R}$. Then we obtain a quasimorphism $q'^*\mu_{\rho_{\genus}}\in \QQQ(\G')^{\G}$.
This approach of employing $q\colon \G\to \R$ was suggested to the authors by Morimichi Kawasaki.

\begin{proof}[Proof of Theorem \textup{\ref{thm:main}}]
  We set
  \begin{align}\label{eq:y_i,z_i}
    y_i = {}^{b_i a_i^2} b_i^{-1}, \  \ z_i = {}^{b_i a_i^2} a_i^{-\genus} \in \G
  \end{align}
  for every $i=1,2,\ldots ,\genus$. In what follows, we will verify that conditions (i), (ii) and (iii) in Lemma~\ref{lem:hojo} are satisfied for $k=\genus$, $y_1,\ldots ,y_{\genus}$, $z_1,\ldots ,z_{\genus}$ and $\mu=q'^*\mu_{\rho_{\genus}}$. Condition (i) trivially holds since $N=\G'$. To verify condition~(ii), we first claim that
\begin{equation}\label{eq:hairu}
[y_1,{}^{b_1 a_1^2} a_1^{-1}]\cdots [y_{\genus},{}^{b_{\genus} a_{\genus}^2} a_{\genus}^{-1}]\in [\G,\G']
\end{equation}
To show \eqref{eq:hairu}, set $\overline{g}=g[\G,\G']\in \G/[\G,\G']$ for $g\in \G$. Then, as we argued in the proof of Lemma~\ref{lem:[G,N]}, we have $\overline{gw}=\overline{wg}$ and (in particular) $\overline{{}^gw}=\overline{w}$ for every $g\in \G$ and $w\in \G'$. Also, recall the surface group relation $[a_1, b_1] \cdots [a_{\genus}, b_{\genus}] = 1_{\G}$. Thus, we obtain the following equalities:
\begin{align*}
\overline{[y_1,{}^{b_1 a_1^2} a_1^{-1}]\cdots [y_{\genus},{}^{b_{\genus} a_{\genus}^2} a_{\genus}^{-1}]}=& \overline{[y_1,{}^{b_1 a_1^2} a_1^{-1}]}\cdots \overline{[y_{\genus},{}^{b_{\genus} a_{\genus}^2} a_{\genus}^{-1}]}\\
=& \overline{{}^{b_1 a_1^2}[b_1^{-1}, a_1^{-1}]}\cdots \overline{{}^{b_{\genus} a_{\genus}^2}[b_{\genus}^{-1}, a_{\genus}^{-1}]}\\
=& \overline{[b_1^{-1}, a_1^{-1}]}\cdots \overline{[b_{\genus}^{-1}, a_{\genus}^{-1}]}\\
=& \overline{[a_1, b_1]}^{-1}\cdots \overline{[a_{\genus},b_{\genus}]}^{-1}\\
=& \overline{[a_1, b_1]\cdots [a_{\genus},b_{\genus}]}^{-1}\\
=& \overline{1_{\G}}^{-1}=\overline{1_{\G}},
\end{align*}
hence obtaining \eqref{eq:hairu}. Then by Lemma~\ref{lem:[G,N]}, we conclude that condition (ii) holds.

Finally, we verify condition~(iii). Let $n\in \NN$.  Since
  \begin{align*}
    q'^*\mu_{\rho_{\genus}} ([y_1, z_1^n] \cdots [y_\genus, z_{\genus}^n]) &= \mu_{\rho_{\genus}}([y,z^n]^{\genus}) = \genus \cdot \mu_{\rho_{\genus}}([y, z^n]),
  \end{align*}
  we obtain
  \begin{align}\label{eq:big}
    \left| q'^*\mu_{\rho_{\genus}} ([y_1, z_1^n] \cdots [y_\genus, z_{\genus}^n])\right| \geq \genus (n(\genus - 1) - 1)
  \end{align}
by \eqref{single_comm_m}. Inequality \eqref{eq:big} in particular yields
\[
\lim_{n\to \infty} \left| q'^*\mu_{\rho_{\genus}} ([y_1, z_1^n] \cdots [y_\genus, z_{\genus}^n])\right|=\infty,
\]
verifying condition~(iii). Therefore Lemma~\ref{lem:hojo} applies, and the sequence
\begin{equation}\label{eq:ore_no_x}
(x_n)_{n\in \NN}=([y_1,z_1^{n}]\cdots [y_{\genus},z_{\genus}^{n}])_{n\in \NN}
\end{equation}
witnesses the non-equivalence of $\scl_{\G}$ and $\scl_{\G,\G'}$.
\end{proof}

\begin{remark}\label{rem:non-ext_qmu_surf}
  The $\G$-invariant homogeneous quasimorphism $q'^*\mu_{\rho_{\genus}} \in \QQQ(\G')^{\G}$ defines a non-zero element of $\W(\G, \G')$.
  Indeed, if not, then we may take $h \in \HHH^1(\G')^{\G}$ and $\mu \in \QQQ(\G)$ such that $q'^*\mu_{\rho_{\genus}} = h + i^* \mu$.
  Then we have
\[
\sup_{n\in \NN}| q'^*\mu_{\rho_{\genus}}(x_n)|=\sup_{n\in \NN}| \mu(x_n)|\leq (2\genus -1)D(\mu)<\infty,
\]
contradicting $\lim\limits_{n\to \infty}| q'^*\mu_{\rho_{\genus}}(x_n)|=\infty$. Here, note that $h(x_n)=0$ since $x_n\in [\G,\G']$.
\end{remark}

\section{Proof of Theorem \ref{thm:main3}}\label{sec:thm3}

We recall our setting from the introduction as follows.
Let $X$ be a hyperbolic $3$-manifold that fibers over the circle with fiber $\Sigma_{\genus}$.
Then there exists $\psi \in \Aut_{+}(\G)$ such that
 $\pi_1(X) = \G \rtimes_{\psi} \ZZ$. Using this $\psi$, we set $\Gps=\pi_1(X)$. The abelianization map $\mathrm{Ab}_{\G}\colon \G\to \ZZ^{2\genus}$ induces a surjection $\pps \colon \Gps \to \ZZ^{2\genus} \rtimes_{s_{\genus}(\psi)} \ZZ$, where $s_{\genus}\colon \Aut_+(\G) \to \Sp(2\genus,\ZZ)$ is the symplectic representation, which is induced by $\mathrm{Ab}_{\G}$.
Set $ \Nps=\Ker(\pps)$; we have $\Nps=\iota(\G')$, where $\iota\colon \G\to \Gps=\G \rtimes_{\psi} \ZZ$ is the natural inclusion. Set the group quotient $\Gamps=\Gps/\Nps \simeq \ZZ^{2\genus} \rtimes_{s_{\genus}(\psi)} \ZZ$.
It was essentially shown in the proof of Theorem 1.2 of \cite{non-extendable} that the pullback $\HHH^2(\Gamps;\RR) \to \HHH^2(\Gps;\RR)$ of the quotient map is surjective.
Hence every representation $\Gps \to \h$ satisfies the condition in
Construction \ref{construction}.

Let $\iota \colon \G \to \Gps$ be the natural inclusion.
It is known that there exists a representation $\overline{\rho} \colon \Gps \to \h$ whose restriction $\rho = \overline{\rho} \circ \iota \colon \G \to \PSL(2, \RR) \hookrightarrow \h$ is a Fuchsian representation (see \cite[Section 7.1]{MR3887602} for example).
We apply Construction \ref{construction} to this representation $\overline{\rho}$ and obtain a $\Gps$-invariant homogeneous quasimorphism $\mu_{\overline{\rho}} \colon \Nps \to \RR$.

Since the pullback $\HHH^2(\G/\G';\RR) \to \HHH^2(\G;\RR)$ is also surjective, the Fuchsian representation $\rho \colon \G \to \h$ fulfills the assumption in Construction \ref{construction}.
Let $\mu_{\rho} \colon \G' \to \RR$ be a resulting $\G$-invariant homogeneous quasimorphism.
By condition (2) of Theorem \ref{thm:formula}, the pullback $\iota^*\mu_{\overline{\rho}}$ coincides with $\mu_{\rho}$ on $[\G, \G']$.
Theorem \ref{thm:intro_non-ext}, together with the fact that the class $\rho^*e_{\RR}$ is non-zero, the quasimorphism $\mu_{\rho}$ defines a non-zero element of $\W(\G, \G')$.

The quasimorphism $q'^*\mu_{\rho_{\genus}}$ used in the proof of Theorem \ref{thm:main} also defines a non-zero element of $\W(\G, \G')$ by Remark \ref{rem:non-ext_qmu_surf}.
By using the following theorem, we can compare these two quasimorphisms up to $\HHH^1(\G')^{\G} + i^*\QQQ(\G)$.

\begin{theorem}[{\cite[Theorem~1.1]{non-extendable}}]\label{thm:one-dim_surf}
We have
\[
\dim_{\RR} \W(\G, \G')=1.
\]
\end{theorem}

By Theorem \ref{thm:one-dim_surf}, there exist $h \in \HHH^1(\G')^{\G}$, $\mu \in \QQQ(\G)$, and a non-zero constant $a$ such that
\begin{align*}
  \mu_{\rho} = a \cdot (q'^*\mu_{\rho_{\genus}}) + h + i^*\mu.
\end{align*}

\begin{lemma}\label{lem:comparison}
For the non-zero constant $a$ above, the inequality
  \[
    | \mu_{\rho}(x_n) | \geq |a| \cdot (\genus (n(\genus - 1) - 1)) - (2\genus - 1) D(\mu)
  \]
  holds for every positive integer $n$.
  Here $x_n$ is the element of $\G$ defined by  \eqref{eq:ore_no_x} and \eqref{eq:y_i,z_i}.
\end{lemma}

\begin{proof}
  Let $n\in \NN$. We have shown that $x_n\in [\G, \G']$ in the proof of Theorem~\ref{thm:main}.
  Since $h$ is a $\G$-invariant homomorphism, we have $h(x_n)=h([y_1, z_1^n] \cdots [y_\genus, z_\genus^n]) = 0$.
  Hence we obtain
\[
  \mu_{\rho}(x_n) = a\cdot q'^*\mu_{\rho_{\genus}}(x_n)+\mu(x_n).
\]
Since $|\mu(x_n)|=|\mu([y_1, z_1^n] \cdots [y_\genus, z_\genus^n])| \leq (2\genus - 1)D(\mu)$, we obtain
  \[
    | \mu_{\rho}(x_n) | \geq |a| \cdot |q'^*\mu_{\rho_{\genus}}(x_n) | - (2\genus - 1) D(\mu).
  \]
  Hence \eqref{eq:big} completes the proof.
\end{proof}

We are now ready to prove Theorem \ref{thm:main3}.

\begin{proof}[Proof of Theorem \textup{\ref{thm:main3}}]
Let $y_1,\ldots ,y_{\genus}$ and $z_1,\ldots ,z_{\genus}$ are the elements of $\G$ defined in \eqref{eq:y_i,z_i}. Let $x_n\in \G$ be defined in \eqref{eq:ore_no_x}.
  We set $\eta_i = \iota(y_i)$ and $\zeta_i = \iota(z_i)$ for every $i=1,\ldots ,\genus$ and
  \[
    \xi_n=[\eta_1,\zeta_1^n]\cdots[\eta_{\genus},\zeta_{\genus}^n]\quad(=\iota(x_n))
  \]
  for every $n\in \NN$.
Take the quasimorphism $\mu_{\overline{\rho}}\in \QQQ(\Nps)^{\Gps}$ constructed in the first part of this section.
In what follows, we will prove that conditions (i), (ii) and (iii) in Lemma~\ref{lem:hojo} are fulfilled for $k=\genus$, $y_1=\eta_1,\ldots ,y_{\genus}=\eta_{\genus}$, $z_1=\zeta_1,\ldots ,z_{\genus}=\zeta_{\genus}$, and $\mu=\mu_{\overline{\rho}}$. Condition (i) is clear: recall that $\Nps=\iota(\G')$.  Since $x_n \in [\G, \G']$, we have $\xi_n\in [\iota(\G),\Nps]\subset [\Gps,\Nps]$. Hence condition (ii) holds. Finally, we discuss condition (iii).
Note that $\iota^*\mu_{\overline{\rho}} = \mu_{\rho}$ on $[\G, \G']$ by Theorem \ref{thm:formula}~(2).
Hence Lemma \ref{lem:comparison} implies that
  \begin{align*}
    |\mu_{\overline{\rho}}(\xi_n)| 
    &= |\mu_{\rho}(x_n)|\\
    & \geq |a| \cdot (\genus (n(\genus - 1) - 1)) - (2\genus - 1) D(\mu).
  \end{align*}
By recalling that $a\ne 0$, we conclude that
\[
\lim_{n\to \infty}|\mu_{\overline{\rho}}(\xi_n)|=\infty,
\]
thus verifying condition (iii). Therefore, Lemma~\ref{lem:hojo} applies, and the sequence $(\xi_n)_{n\in \NN}$ witnesses the non-equivalence of $\scl_{\Gps}$ and $\scl_{\Gps,\Nps}$.
\end{proof}

\begin{remark} \label{remark:equivalence}
In \cite{non-extendable}, Kawasaki, Kimura and the authors showed that in several cases $\scl_G$ and $\scl_{G,N}$ are equivalent. Their results are summarized as follows: if
\begin{align} \label{equation:equivalence}
\W(G,N) = 0,
\end{align}
then $\scl_G$ and $\scl_{G,N}$ are equivalent; moreover if $\Gamma = G/N$ is solvable, then $\scl_G$ and $\scl_{G,N}$ coincide on $[G,N]$ (\cite[Theorem~2.1]{non-extendable}).
In \cite{non-extendable}, they also showed that \eqref{equation:equivalence} holds for the pair $(G,N)=(G,G')$, where $G$ is a free group or the fundamental group of a non-orientable closed connected surface. Therefore, for these examples, $\scl_G$ and $\scl_{G,N}$ coincide on $[G,N]$.
As we stated as Theorem~\ref{thm:one-dim_surf} (\cite[Theorem~1.1]{non-extendable}), the dimension of $\W(\G, \G')$ equals $1$ for $\genus \ge 2$. Therefore Theorem~\ref{thm:main}
implies that $\scl_G$ and $\scl_{G,N}$
can be \emph{non}-equivalent even if the space $\W(G,N)$ is only $1$-dimensional.

We say that a surjective group homomorphism $p \colon G \to \Gamma$ \emph{virtually splits} if there exists a subgroup $\Lambda$ of finite index of $\Gamma$ and a group homomorphism $s \colon \Lambda \to G$ satisfying $p(s(\gamma)) = \gamma$ for every $\gamma \in \Lambda$.
In \cite[Theorem~1.5]{MR4452430}, Kawasaki--Kimura--Matsushita--Mimura showed that $\scl_{G,N} \le 2 \cdot \scl_G$ on $[G,N]$ if the projection $G \to G/N$ virtually splits. In this case, we have $\QQQ(N)^G = i^* \QQQ(G)$; see also \cite[Proposition~1.4]{MR4425357}.
\end{remark}

\section{Concluding remarks}\label{sec:remarks}

\subsection{Non-extendable quasimorphisms and taut foliations} \label{sec:taut_foliation}
Let $X$ be a hyperbolic $3$-manifold that fibers over the circle with fiber $\Sigma_{\genus}$.
Let $\Gps$, $\Nps$, and $\psi \in \Aut_+(\G)$ be as in Theorem \ref{thm:main3}.
Let $i \colon \Nps \to \Gps$ be the inclusion.
It was shown in \cite[Theorem~1.2]{non-extendable} that the space $\W(\Gps, \Nps)$ is isomorphic to $\HHH^2(\Gps;\RR)$, and the dimension of them is equal to $1 + \dim(\Ker (I_{2\genus} - s_\genus(\psi))$.
Here $I_{2\genus}$ denotes the identity matrix of size $2\genus$.

We used the element $\mu_{\overline{\rho}}$ of $\QQQ(\Nps)^{\Gps}$ in the proof of Theorem \ref{thm:main3}, which is non-zero in $\W(\Gps, \Nps)$.
This $\Gps$-invariant homogeneous quasimorphism is from Construction \ref{construction} applied to the representation $\overline{\rho} \colon \Gps \to \h$.
This $\overline{\rho}$ is (the most basic) one of the universal circle representations due to Thurston, that is, the representation corresponds to the taut foliation on $X$ whose leaves are the fibers.
We refer to \cite{MR1965363}, \cite{MR2327361}, and \cite{MR4264583} for taut foliations and universal circle representations.

For a taut foliation $\mathcal{F}$ on $X$, let $\rho_{\mathcal{F}} \colon \Gps \to \h$ be the universal circle representation defined via $\mathcal{F}$.
As was mentioned in Section \ref{sec:thm3}, the pullback $\HHH^2(\Gamps;\RR) \to \HHH^2(\Gps;\RR)$ is surjective. (Recall that $\Gamps=\Gps/\Nps$ and that it is solvable.)
Hence every representation $\rho_{\mathcal{F}}$ gives rise to a $\Gps$-invariant homogeneous quasimorphism $\mu_{\rho_{\mathcal{F}}}$ on $\Nps$ via Construction \ref{construction}.
Here, note that $\mu_{\rho_{\mathcal{F}}}(= \nu_{\rho_{\mathcal{F}, A, u}})$ itself does depend also on the choices of cochains $A$ and $u$.
However, as we discussed in Remark 3.5, the class $[\mu_{\rho_{\mathcal{F}}}]$ in $\QQQ(N_{\psi})^{G_{\psi}}/\HHH^1(N_{\psi};\mathbb{R})^{G_{\psi}}$ is uniquely determined by $\rho_{\mathcal{F}}$.
Now we define $\mu_{\mathcal{F}}$ as the class $[\mu_{\rho_{\mathcal{F}}}]$ in $\W(\Gps, \Nps)$, which depends only on $\mathcal{F}$.
We set $\mu_{\mathcal{F}} = [\mu_{\rho_{\mathcal{F}}}] \in \W(\Gps, \Nps)$.

The goal of this section is to show the following:
\begin{proposition}\label{prop:taut_qm}
  The space $\W(\Gps, \Nps)$ is $\RR$-spanned by
  \[
    \{ \mu_{\mathcal{F}} \mid \mathcal{F} \text{ is a taut foliation on } X \}.
  \]
\end{proposition}

\begin{proof}
  Let us consider the following diagram whose rows are exact:
  \[
  \xymatrix{
  \HHH^1(\Nps;\RR)^{\Gps} \ar[r] & \QQQ(\Nps)^{\Gps} \ar[r]^-{\mathbf{d}} & \HHH_b^2(\Nps;\RR)^{\Gps} & \\
  & \QQQ(\Gps) \ar[r]^-{\mathbf{d}} \ar[u]^-{i^*} & \HHH_b^2(\Gps;\RR) \ar[r]^{c_{\Gps}} \ar[u]^{i^*} & \HHH^2(\Gps;\RR).
  }
  \]
  Here $\mathbf{d} \colon \QQQ(\Nps)^{\Gps} \to \HHH_b^2(N;\RR)^{\Gps}$ is the map defined by $\mathbf{d}(\mu) = [\delta \mu] \in \HHH_b^2(\Nps;\RR)$.
  Note that $\mathbf{d}(\HHH^1(\Nps;\RR)^{\Gps} + i^*\QQQ(\Gps)) \subset i^*(\HHH_b^2(\Gps;\RR))$ and that $i^* \colon \HHH_b^2(\Gps;\RR) \to \HHH_b^2(\Nps;\RR)^{\Gps}$ is an isomorphism since $\Gamps$ is amenable.
  Since $\Gps$ is Gromov-hyperbolic, the map $c_{\Gps}$ is surjective.
  Hence the map $c_{\Gps} \circ (i^*)^{-1} \circ \mathbf{d}$ induces an injection
  \[
    \W(\Gps, \Nps) = \QQQ(\Nps)^{\Gps}/(\HHH^1(\Nps;\RR)^{\Gps} + i^*\QQQ(\Gps)) \to \HHH^2(\Gps).
  \]

  By the construction of $\mu_{\rho_{\mathcal{F}}}$, we have
  \[
    c_{\Gps} \circ (i^*)^{-1} \circ \mathbf{d}(\mu_{\rho_{\mathcal{F}}}) = \rho_{\mathcal{F}}^*(e_{\RR}).
  \]
  Hence it suffices to show that the cohomology group $\HHH^2(\Gps;\RR)$ is $\RR$-spanned by
  \[
    \{ \rho_{\mathcal{F}}^* (e_{\RR}) \mid \mathcal{F} \text{ is a taut foliation on } X \}.
  \]
This then turns out to be a direct consequence of Gabai's theorem \cite[Remark 7.3]{MR1470750} (see also \cite[Theorem 1.4]{MR4205410}).
  More precisely, let $x^* \colon \HHH^2(X;\RR) \to \RR$ be the dual Thurston norm.
  Since $X$ is hyperbolic, the unit norm ball $B_{x^*}$ of $x^*$ is a polyhedron in $\HHH^2(X;\RR)$ \cite[Theorem 2]{MR823443}.
  Gabai's theorem asserts that every vertex of the dual Thurston norm ball for a compact oriented irreducible $3$-manifold can be realized as the Euler class of some taut foliation.
  Hence, for every vertex $v \in B_{x^*}$, there exists a taut foliation $\mathcal{F}$ such that the Euler class $e(T\mathcal{F})$ of the plane field of $\mathcal{F}$ is equal to $v$.
  Hence the cohomology group $\HHH^2(X;\RR)$ is $\RR$-spanned by
  \[
    \{ e(T\mathcal{F}) \mid \mathcal{F} \text{ is a taut foliation on } X \}.
  \]
  It is known that $\rho_{\mathcal{F}}^* (e_{\RR}) \in \HHH^2(\Gps)$ is equal to $e(T\mathcal{F}) \in \HHH^2(X;\RR)$ under the canonical isomorphism $\HHH^2(\Gps;\RR) \cong \HHH^2(X;\RR)$ (see \cite[Proposition 7.1]{MR4029686}).
  This completes the proof.
\end{proof}

\subsection{Overflow argument}\label{subsec:overdrive}
In this short subsection, we provide an argument to show the non-equivalence of $\scl_G$ and $\scl_{G,N}$. In the proofs of Theorems~\ref{thm:main}, \ref{thm:main3} and \ref{thm:main2}, we have employed the auxiliary lemma (Lemma~\ref{lem:hojo}) to show the non-equivalence of $\scl_G$ and $\scl_{G,N}$.
Conditions (i) and (ii) in Lemma~\ref{lem:hojo} are purely algebraic conditions. Contrastingly, condition (iii) is formulated in terms of the behavior of  values $\mu(x_n)$ as $n\to \infty$, where $(x_n)_{n\in \NN}=([y_1,z_1^n]\cdots [y_k,z_k^n])_{n\in \NN}$.
Given $\mu$, it seems to be a hard task to find elements $y_1,\ldots,y_k,z_1,\ldots,z_k$ that fulfill condition (iii).
 Nevertheless, the following argument, which we call the \emph{overflow argument,} allows us to derive condition (iii)
by
estimating the value of a quasimorphism at a \emph{single} element.

\begin{lemma}[overflow argument] \label{overflow}
  Let $G$ be a group, $N$ a normal subgroup, and let $\mu \in \QQQ(N)^G$.
Let $k$ be a positive integer and let $y_1, \dots, y_k, z_1, \dots, z_k \in G$. Assume that the following three conditions are fulfilled:
  \begin{enumerate}[$(i)$]
    \item $[y_i, z_i] \in N$ for every $i = 1, \dots, k$;
    \item $[y_1, z_1] \cdots [y_k, z_k] \in [G, N]$;
    \item[$(iii)'$] $|\mu([y_1, z_1] \cdots [y_k, z_k])| > (2k-1)D(\mu)$.
  \end{enumerate}
Then we have
\begin{equation}\label{eq:iii}
\lim_{n \to \infty} \mu ([y_1, z_1^n] \cdots [y_k, z_k^n]) = \infty.
\end{equation}
In particular, $\scl_G$ and $\scl_{G,N}$ are \emph{not} equivalent.
\end{lemma}
\begin{proof}
By the definition of homogeneous quasimorphisms and $G$-invariance of $\mu$, we have
  \begin{align*}
    \mu([y_1, z_1^n] \dots [y_k, z_k^n]) & \underset{(k-1)D(\mu)}{\sim} \mu([y_1, z_1^n]) + \cdots + \mu([y_k, z_k^n])\\
    & \underset{k(n-1)D(\mu)}{\sim} n (\mu([y_1, z_1] + \cdots + \mu([y_k, z_k]))\\
    & \underset{n(k-1)D(\mu)}{\sim} n\mu([y_1, z_1]\dots [y_k, z_k]).
  \end{align*}
  Hence we obtain
  \begin{align*}
    |\mu([y_1, z_1^n] \dots [y_k, z_k^n]) - n\mu([y_1, z_1]\dots [y_k, z_k])| \leq (n(2k-1)-1)D(\mu).
  \end{align*}
  This, together with (iii)$'$, implies \eqref{eq:iii}. 
Now the latter assertion follows from Lemma~\ref{lem:hojo} since \eqref{eq:iii} is exactly condition (iii) in Lemma~\ref{lem:hojo}.
\end{proof}

\subsection{Rigidity of the invariant quasimorphisms}\label{sec:rigid}

Let $N$ be a normal subgroup of $G$ contained in $G'$.
In Remark \ref{rem:indep}, we showed that the map
\[
  \Phi \colon \Hom(G, \h)_{\basic} \to \QQQ(N)^G/\HHH^1(N;\RR)^G
\]
sending $\rho$ to $[\mu_{\rho}]$ is well-defined.
The goal of this subsection is to show the rigidity of the map $\Phi$ (Proposition \ref{prop:rigidity}).

To state Proposition \ref{prop:rigidity}, we recall the notion of semi-conjugacy introduced in \cite{MR893858}.
The following definition of semi-conjugacy can be found in \cite[Theorem 1.4 (iv)]{MR3692890}.
\begin{definition}\label{def:semiconj}
  Two elements $\rho_1, \rho_2$ in $\Hom(G,\h)$ are \textit{semi-conjugate} if there exist a monotone map $\tf \colon \RR \to \RR$ with $\tf(x + 1) = \tf(x) + 1$ for every $x \in \RR$ and lifts $\trho_1(g), \trho_2(g) \in \th$ of $\rho_1(g), \rho_2(g) \in \h$ for every $g \in G$ such that
  \[
    \trho_1(g)(\tf(x)) = \tf (\trho_2(g)(x))
  \]
  for every $g \in G$ and $x \in \RR$.
\end{definition}
There are other non-equivalent formulations of semi-conjugacy (see \cite{MR3692890}, \cite{MR3888699}, \cite{MR3887602} and the references therein).
Semi-conjugacy in Definition \ref{def:semiconj} is an equivalence relation on $\Hom(G, \h)$, and fits in the following theorem of Ghys.
\begin{theorem}[{\cite{MR893858}, see also \cite[Theorem 1.4]{MR3692890}}] \label{thm:Ghys}
  Let $G$ be a group and $\rho_1$ and $\rho_2$ be elements of $\Hom(G, \h)$.
  Let $e_b \in \HHH_b^2(\h;\ZZ)$ be the bounded Euler class.
  Then, $\rho_1$ is semi-conjugate to $\rho_2$ if and only if $\rho_1^* e_b$ is equal to $\rho_2^* e_b$ in $\HHH_b^2(G;\ZZ)$.
\end{theorem}

Now we are ready to state the main result in this subsection.

\begin{proposition}\label{prop:rigidity}
  Let $\rho_1$ and $\rho_2$ be elements of $\Hom(G, \h)_{\basic}$.
  If $\rho_1$ is semi-conjugate to $\rho_2$, then $\Phi(\rho_1) = \Phi(\rho_2)$.
  In particular, the map $\Phi$ descends to
  \[
    \Phi \colon \Hom(G, \h)_{\basic}/\text{semi-conjugacy} \to \QQQ(N)^G/\HHH^1(N;\RR)^G.
  \]
\end{proposition}


\begin{proof}
  The exact sequence $0 \to \HHH^1(N;\RR)^G \to \QQQ(N)^G \to \HHH_b^2(N;\RR)^G$ implies that the map
  \[
    \QQQ(N)^G/\HHH^1(N)^G \to \HHH_b^2(N;\RR)^G
  \]
  is injective.
  Let us consider the following diagram:
  \begin{align}\label{Phi_diagram}
  \xymatrix{
  \Hom(G;\h)_{\basic} \ar[rr]^-{\Phi} \ar[d]^{e} & & \QQQ(N)^G/\HHH^1(N;\RR)^G \ar[d]^-{\mathbf{d}} \\
  \HHH_b^2(G;\ZZ) \ar[r]^c & \HHH_b^2(G;\RR) \ar[r]^-{i^*} & \HHH_b^2(N;\RR)^G.
  }
  \end{align}
  Here $i^*$ is the map in (\ref{bdd_five_term}), $c$ is the change of coefficients homomorphism, $e$ is the map sending $\rho$ to $\rho^* e_b$, and $\mathbf{d}$ is the map sending $[\mu]$ to $[\delta \mu] \in \HHH_b^2(N;\RR)^G$.
  Note that the map $\mathbf{d}$ is injective since the following is exact:
  \[
    0 \to \HHH^1(N;\RR)^G \to \QQQ(N)^G \to \HHH_b^2(N;\RR)^G.
  \]
  We now show that diagram (\ref{Phi_diagram}) commutes.
  Let $\rho$ be an element of $\Hom(G;\h)$.
  The element $\Phi(\rho)$ is represented by the homogeneous quasimorphism $\nu_{\rho, A, u}$ in Construction \ref{construction}.
  Since $\delta \nu_{\rho, A, u} = (\rho^* \chi_b - p^* A)|_{N} = (\rho^* \chi_b)|_{N}$ and $\chi_b$ is a cocycle of the real bounded Euler class, we have
  \[
    \mathbf{d} \circ \Phi(\rho) = i^* \rho^* [\chi_b] = i^* \circ c \circ e(\rho).
  \]

  Assume that $\rho_1 \in \Hom(G;\h)_{\basic}$ is semi-conjugate to $\rho_2$.
  Then, by Theorem \ref{thm:Ghys} and the commutativity of (\ref{Phi_diagram}), we have
  \[
    \mathbf{d} \circ \Phi(\rho_1) = i^* \circ c \circ e (\rho_1) = i^* \circ c \circ e (\rho_2) = \mathbf{d} \circ \Phi(\rho_2).
  \]
  Since $\mathbf{d}$ is injective, we have that $\Phi(\rho_1) = \Phi(\rho_2)$.
\end{proof}

\begin{remark}
  If $G = \pi_1(\Sigma_{\genus})$ of genus $\genus \geq 2$, then the space $\Hom(G,\h)$ is equal to the space of all representations $\Hom(G,\h)$ from $G$ to $\h$.
\end{remark}

\section*{Acknowledgements}
The authors would like to express their deepest gratitude to Morimichi Kawasaki, who suggested us employing surjections from $\G$ to $\R$ as a possible strategy for  establishing Theorem~\ref{thm:main}. The authors thank Mitsuaki Kimura, Yoshifumi Matsuda, Shigenori Matsumoto, Yoshihiko Mitsumatsu, Hiraku Nozawa for discussions and comments.

The first author is supported by JSPS KAKENHI Grant Number JP21J11199.
The second author and the third author are partially supported by JSPS KAKENHI Grant Number JP19K14536 and JP21K03241, respectively.

\bibliographystyle{amsalpha}
\bibliography{references.bib}
\end{document}